\documentclass{amsart}

\usepackage[
  hmarginratio={1:1},     
  vmarginratio={1:1},     
  textwidth=400pt,        
  heightrounded,          
]{geometry}
\usepackage{xcolor}
\usepackage{amsfonts}
\usepackage{mathrsfs}
\usepackage{fancyhdr}
\usepackage{amsmath,amsthm,amssymb}
\usepackage{enumerate}
\usepackage{hyperref}
\usepackage{array}
\usepackage{tabu}
\usepackage [english]{babel}
\usepackage{graphicx}
\usepackage [autostyle, english = american]{csquotes}
\usepackage[all,cmtip]{xy}
\usepackage{tikz-cd}
\usetikzlibrary{babel}
\usepackage{mathtools}
\usepackage[backend=biber, style=numeric]{biblatex}
\graphicspath{ {./images/}  }
\DeclareFieldFormat{postnote}{#1}
\MakeOuterQuote{"}
\newcommand\ii\item
\newcommand\ol\overline

\newcommand\CC{\mathbb C}

\newcommand\RR{\mathbb R}
\newcommand\ZZ{\mathbb Z}

\newcommand{\be}{\begin{enumerate}} 
\newcommand{\bii}{\begin{itemize}}
\newcommand{\bea}{\begin{enumerate}[(a)]}
\newcommand{\ee}{\end{enumerate}}
\newcommand{\eii}{\end{itemize}}
\newcommand\Vect{\mathrm{Vect}}

\newcommand\ooCat{(\infty,1)\mathrm{Cat}}

\newcommand\End{\mathrm{End}}
\newcommand\id{\mathrm{Id}}
\newcommand\C{\mathcal C}

\newcommand\Pshv{\mathrm{Pshv}}
\newcommand\pt{\mathrm{pt}}

\newcommand\Hom{\mathrm{Hom}}

\newcommand\Coh{\mathrm{Coh}}
\newcommand\IndCoh{\mathrm{IndCoh}}
\newcommand\QCoh{\mathrm{QCoh}}
\newcommand\Perf{\mathrm{Perf}}
\newcommand\supp{\mathrm{supp}}

\newcommand\A{\mathbb{A}}
\newcommand\GG{\mathbb{G}}
\newcommand\CCx{\mathbb{C}^\times}
\newcommand\MF{\mathrm{MF}}
\newcommand\musupp{\mu\mathrm{supp}}

\newcommand\res{\mathrm{res}}
\newcommand\ind{\mathrm{ind}}

\newcommand\muhom{\mu\mathrm{hom}}
\newcommand\Spec{\mathrm{Spec}}
\newcommand\Ext{\mathrm{Ext}}

\newcommand\unit{\mathrm{unit}}
\newcommand\Sym{\mathrm{Sym}}
\newcommand\ShvCat{\mathrm{ShvCat}}
\newcommand\xwcmod{\mathrm{Mod}_{\mathrm{Perf}(X)}^{wc}(\mathrm{Cat}_\infty^{\mathrm{perf}})}

\newcommand\perfxmod{\mathrm{Mod}_{\mathrm{Perf}(X)}(\mathrm{Cat}_\infty^{\mathrm{perf}})}
\newcommand\perfxamod{\mathrm{Mod}_{\mathrm{Perf}(X\times\mathbb{A}^1)}(\mathrm{Cat}_\infty^{\mathrm{perf}})}
\newcommand\dAff{\mathrm{dAff}}
\newcommand\ooGrpd{\infty\mathrm{Grpd}}

 \makeatletter
\newtheorem*{rep@theorem}{\rep@title}
\newcommand{\newreptheorem}[2]{%
\newenvironment{rep#1}[1]{%
 \def\rep@title{#2 \ref{##1}}%
 \begin{rep@theorem}}%
 {\end{rep@theorem}}}
\makeatother
 
\newtheorem{thm}{Theorem}[section]

\newtheorem{lem}[thm]{Lemma}
\newtheorem{rmk}[thm]{Remark}
\newtheorem{ex}[thm]{Example}
\newtheorem{defn}[thm]{Definition}

\newtheorem{stat}[thm]{Statement}
\newreptheorem{stat}{Statement}
\newreptheorem{thm}{Theorem}

\begin{document}
\title{Modeling Rozansky--Witten Theory with Sheaves of Categories}
\author{Enoch Yiu}

\begin{abstract}
We model Rozansky--Witten theory of $T^*X$ using modules over $\Perf(X\times\A^1)$ viewed as sheaves of categories over $X\times \A^1$.  This is in parallel to Tamarkin's approach to nonconic Lagrangians in $T^*X$ via the singular support of sheaves on $X\times\RR$.  More specifically, we construct the objects of Rozansky--Witten theory of $T^*X$ given by hybrid Lagrangians of which graphs and conormals are special cases.  As a consistency check, we show that $\Hom$s between such objects are given by suitable matrix factorizations.
\end{abstract}

\maketitle

\tableofcontents

\section{Introduction}

\subsection{Contents of this Paper}

In \cite{RW97}, Rozansky and Witten studied a three dimensional topological sigma model with a target being a hyperk\"{a}hler manifold (which is later extended to complex symplectic manifolds \cite{MK}), which gives rise to topological invariants of 3-manifolds, or invariants to complex symplectic manifolds called the weight system, by considering the path integral over the space of all maps from the closed oriented 3-manifolds to the the target \cite{kontsevich}\cite{MK}\cite{RS}.  A three dimensional topological quantum field theory gives rise to a $(\infty,3)$ symmetric monoidal functor from the framed bordism 3-category to a symmetric monoidal 3-category.  Now each complex symplectic manifold gives rise to such a theory, and hence such a functor.  This generalizes the topological invariants of 3-manifolds mentioned above since such a functor assigns a number to a 3-manifold, a vector space to a 2-manifold, a category to a 1-manifold and a 2-category to a 0-manifold.  By the cobordism hypothesis \cite{BD}\cite{LCH}, such a functor is determined by its assignment to a point.  In \cite{KAPUSTIN2009295}, Kapustin, Rozansky and Saulina studied the boundary conditions of Rozansky--Witten theory using path integral analysis and showed that the boundary conditions correspond to complex fibrations over complex Lagrangian submanifolds in the complex symplectic target.  Such boundary conditions forms a 2-category and in \cite{KR} Kapustin and Rozansky suggested a definition of 2-category associated to a complex symplectic manifold, which gives objects that are in the form of a graph of a differential $df$ (and fibrations over them).  It is also suggested that the general objects of the 2-category should be sheaves of categories over complex Lagrangian submanifolds \cite{KAPUSTIN2009295}\cite{KR}.  The Rozansky--Witten theory of $T^*X$ is known to be related to sheaves of categories over $X$ \cite{KAPUSTIN2009295}\cite{KR}.  This is a 3d analogue of the 2d story where the Fukaya category of a cotangent bundle being related to sheaves on the base \cite{NZ}.  The big difference, however, is that for holomorphic branes there are no instantons, so everything in Rozansky--Witten theory is local and for a general complex symplectic target $Y$ one should be able to patch from local cotangent covers.  The rough idea is that by taking singular support, in the sense of Stefanich et al \cite{stefanich}, which is analogous to the definition of singular support in \cite{AG}\cite{BIK}, we would obtain a conical subset in $T^*X$.  Under some notion of ``weakly constructibility" analogous to the singular support of a sheaf \cite{KS}, one would hope to obtain the Lagrangian condition.  In this paper, we construct objects that are "in between" these two cases, namely sheaves of categories in $X\times \A^1$ whose microsupports are hybrid Lagrangians of the form $df + T_Z^*X$, where $Z$ is a smooth subscheme of $X$.  This can be considered as the deformation of the $\Perf(X)$-module category $\Perf(Z)$ whose singular support is the conic Lagrangian $T_Z^*X$ as mentioned in \cite[Lesson 3.2.3]{T}.  One of the guiding motivations of the author is the following statement by Teleman \cite{T}:
\begin{stat}
\label{teleman}
Pure topological gauge theory in 3 dimensions for a compact Lie group $G$ is equivalent to the Rozansky-Witten theory for the BFM space of the Langlands dual Lie group $G^\vee$.
\end{stat}
In the abelian case, this means that categories with a topological torus action, is equivalent to the Rozansky--Witten theory for $T^*(\CCx)^n$, where $(\CCx)^n$ is the dual complexified torus.  Under this equivalence, the trivial representation corresponds to the conormal bundle at $1\in(\CCx)^n$, and the regular representation corresponds to the zero section.  Suppose we are given a category $\mathcal C$ with a torus action, the $\Hom$ of the trivial representation with $\mathcal C$ gives us $\mathcal C^T$, the category of fixed objects, and the $\Hom$ of the regular representation with $\mathcal C$ gives us the underlying category.  So now suppose $\mathcal C$ corresponds to the Lagrangian $L$ in $T^*(\CCx)^n$, the category of fixed objects correspond to the $\Hom$ between the $T_1^*(\CCx)^n$ and $L$, and the underlying category corresponds to the $\Hom$ between the zero section and $L$.  In Rozansky--Witten theory, the $\Hom$ between $L$ and $M$ is supposed to be supported on $L\cap M$, so if $L$ and $M$ are disjoint, the $\Hom$ between them is zero.  Under the above equivalence, this means that given a Lagrangian that does not intersect the zero section but intersects the conormal at $1$, there are categories with nontrivial torus action but with a trivial underlying category.  Part of this story can be described by Fourier--Mukai transform \cite{T}\cite{AD}:

When $G$ is $S^1$, the Fourier--Mukai transform gives an equivalence between local systems of $S^1$ with convolution and quasicoherent sheaves on $\CCx$ with the standard tensor product.  Therefore modules over local systems of $S^1$, i.e. categories with a topological $S^1$ action is equivalent to modules over $\QCoh(\CCx)$, i.e. sheaves of categories over $\CCx$.  To get to the Rozansky--Witten theory of $T^*\CCx$, with objects being Lagrangians in $T^*\CCx$, we take the singular support of the modules over $\QCoh(\CCx)$, we ignore the sizes of the modules for now, which would be discussed in Section 5.  But taking singular support only gives us conic Lagrangians and so the Fourier--Mukai transform described above describes only the conic story.  This paper attempts to extend the conic B-side story to nonconic objects.  In particular, we construct many objects with its nonconic microsupport not intersecting the zero section.

In Section 2, we first go over some needed notions in derived geometry.  The treatment mainly follows \cite{gr}.  Even though the main objects we consider are classical, we need derived schemes because we are taking intersections of subschemes and also the use of sheaves of categories and integral transform formulas of perfect and coherent sheaves as in \cite{bznp2017}.  In Section 3, we go over the definitions of categories of singularities and matrix factorizations as in \cite{Orlov_2003} and \cite{preygel2011thomsebastianidualitymatrix}.  In Section 4, we go over the proposal in \cite{KR}, of modelling the boundary conditions of Rozansky--Witten theory using matrix factorizations.  We also discuss the relationship between topological gauge theories and Rozansky--Witten theory in \cite{T} and give a sketch of a construction of a nontrivial $S^1$ action on the trivial category.  In Section 5, we use the definition of singular support of $\Perf(X)$-modules by Stefanich et al \cite{stefanich} to calculate the singular support of objects of the form $\Perf(Z)$ where $Z$ is a smooth subscheme.  The singular support of such objects is the conormal $T_Z^*X$ and finite colimit of such objects will give a subset of the unions of singular support, so we obtain objects that have subsets of unions of conormals as singular support.  These objects, at least those with Lagrangian singular support since we only know such colimits have isotropic singular support, give a part of the conic part of Rozansky--Witten theory, where the full conic part should be $\Perf(X)$-modules that are ``weakly constructible", i.e. modules that have Lagrangian singular support, which we denote this category by $\mathrm{Mod}_{\Perf(X)}^{wc}(\mathrm{Cat}_\infty^{\mathrm{perf}})$ which is a full subcategory of $\mathrm{Mod}_{\Perf(X)}(\mathrm{Cat}_\infty^{\mathrm{perf}})$.  In Section 6, we extend the conic story to include nonconic objects using Tamarkin's trick.  Given an exact Lagrangian $L$ in $T^*X$, such that the canonical 1-form in $T^*X$ restricts on $L$ as $df$ for some function $f$ which is chosen up to a constant on $L$.  We can construct a conic Lagrangian in $T^*(X\times \A^1)$ by scaling the cofiber directions of $\{(l, f(l),1)\in T^*X\times T^*\A^1 \simeq T^*(X\times \A^1): l\in L\}$.  There's a Hamiltonian action of $\A^1$ on $T^*(X\times \A^1)$ given by translation of the $\A^1$ base coordinate which corresponds to different choices of $f$.  The momentum map $\mu$ of this action is the projection to the $(\A^1)^\vee$ codirection.  Now the Hamiltonian reduction $T^*(X\times\A^1)//\A^1 = \mu^{-1}(1^\vee)/\A^1$ is just $T^*X$ and maps the conic Lagrangians we constructed back to $L$.  We also define the map
$$\rho: T^*X \times T_{\tau\neq 0}^*\A^1 \rightarrow T^*X, \quad (x,\xi,y,\tau) \mapsto (x,\frac{\xi}{\tau}),$$
so that for modules $M$ whose singular support does not lie in $\tau = 0$, we define the microsupport to be $\musupp(M) = \rho(SS(M)\cap \{\tau \neq 0\})$.  Notice that in our construction of conic Lagrangian from $L$, we scaled the cofiber directions, so having $\xi / \tau$ is the same as restricting to $\tau = 1$ as taking the inverse image of $1^\vee$ in Hamiltonian reduction.  Using this definition of microsupport we constructed objects of the form
$$C_{f,Z} = \Perf(y+f = 0, Z \, \,\mathrm{smooth}\,\,\mathrm{subscheme}),$$
with microsupport $df + T_Z^*X$ which are objects in the category
$$\mathcal C^{pre}(X) = \mathrm{Mod}_{\Perf(X\times\A^1)}^{wc}(\mathrm{Cat}_\infty^{\mathrm{perf}}) / \{M: SS(M)\subset \{\tau = 0\}\}.$$
The Hom's of such objects are then calculated to be
\begin{repthm}{hom_thm}
$$\Hom_{\mathcal C^{pre}(X)}(C_{f,Y},C_{g,Z}) = \Coh((Y\times_X Z)_0) / \Perf((Y\times_X Z)_0),$$
the category of singularities of the fiber of $g-f$ at the (derived) intersection at zero.
\end{repthm}
Now notice that we have not mimicked the whole process of Hamiltonian reduction, we have yet to quotient out $\A^1$.  What we instead do is to take $\mathcal C(X)$ to be the objects $C_{f,Z}$ and the finite colimit of such objects.  We define the $\Hom$ to be
$$\Hom_{\mathcal C(X)}(C_{f,Y},C_{g,Z}) = \bigoplus_{s\in k}\Hom_{\mathcal C^{pre}(X)}(C_{f,Y},C_{g+s,Z}),$$
and so we have
\begin{repthm}{mf_thm}
$$\Hom_{\mathcal C(X)}(C_{f,Y},C_{g,Z}) = \mathrm{MF}(Y\times_X Z, g-f).$$
\end{repthm}
When $Y\times_X Z$ is smooth, this category coincides with Kontsevich's defintion of category of B-branes \cite{Orlov_2003}.  In the simpler case where $Y=Z=X$, we recover the expected result in Rozansky--Witten theory
\begin{repthm}{mf_thm_simple}
$$\Hom_{\mathcal C(X)}(C_{f,X},C_{g,X}) = \mathrm{MF}(X, g-f).$$
\end{repthm}
The definition of $\Hom$ is inspired by orbit categories \cite{keller2005}\cite{fan2025dgenhancedorbitcategories}\cite{Chen_2024}.  Under such a $\Hom$, the objects $C_{f,Z}$ and $C_{f+s,Z}$ are identified.  $\mathcal C(X)$ is obviously not the full nonconic story, but are only a part of it.  Since $\A^1$ acts on $X\times \A^1$ by translation, we have an action
$$\Coh_c(\A^1)\otimes \Perf(X\times\A^1)\rightarrow \Perf(X\times\A^1),$$
and hence also an action its modules, where $\Coh_c(\A^1)$ is the category of coherent sheaves on $\A^1$ with finite support with convolution as its monoidal product.  A larger part of the story should be described by the coinvariants construction given by
$$\mathcal C^{pre}(X) \otimes_{\Coh_c(\A^1)} \mathrm{f.d.}\,\Vect.$$
The Hom construction above given by direct sum which is really a colimit would then be the Hom in the coinvariants restricted to the objects mentioned above.  Note that this is still not the full story as this process can only obtain exact Lagrangians, the Lagrangian $dx/x$ is not exact in $T^*\CCx$ and we would not be able to obtain it locally since K\"ahler differentials are not locally exact.

\subsection{Related Works}

Our work has a lot of similarities with \cite{benzvi2025potentcategoricalrepresentations}.  In this work, Ben-Zvi and Nadler considered potent categorical $G$-representations, which are cyclic ind-coherent sheaves of categories in the loop space $\mathcal L BG = G/G$.  One thing to note is that the $G$-representations considered here and the corresponding statements in this work regarding Statement 1.1.1 are ind-coherent sheaves of categories and not actual categories with a $G$ action.  Nevertheless, the authors give a mirror equivalence in the case of abelian $G$ and discusses the case for general gauge group $G$, by modeling the suitable $G$-representations to be cyclic ind-coherent sheaves of categories in the space $T[-1]BG = \mathfrak g/G$.  One difference between our work and \cite{benzvi2025potentcategoricalrepresentations} is the sizes of categories we work with.  In \cite{benzvi2025potentcategoricalrepresentations}, the authors work with ind-coherent sheaves of categories $2\IndCoh$, which is based on the work of Stefanich \cite{GS}.  This category $2\IndCoh$ is not determined by global sections contrary to quasicoherent sheaves of categories \cite{G2015} and is not the category of modules over a monoidal category.  In our work, on the other hand, we mainly work with $\mathrm{Mod}_{\Perf(X)}(\mathrm{Cat}_\infty^{\mathrm{perf}})$, which is a large category of small categories.  Besides the size difference, in constructing our category $\mathcal C(X)$, our idea is to first microlocalize away from the 0-section of the $T^*\A^1$ component, then look at the coinvariants.  In \cite{benzvi2025potentcategoricalrepresentations}, the authors first construct the periodic base
$$\mathcal A = 2\IndCoh(\GG_a) / 2\QCoh(\GG_a),$$
then periodize the indcoherent sheaves of categories by tensoring with $\mathcal A$:
$$2\IndCoh^\pi(X) = 2\IndCoh(X)\otimes \mathcal A$$
regarding as an $\mathcal A$ module.  This construction is the same as microlocalizing away from the 0-section, however instead of modding out the $\GG_a$ action, the authors remembers the $\GG_a$ action by working with the periodization as an $\mathcal A$-module with the convolution symmetric monoidal structure of $\mathcal A$ coming from the the addition in $\GG_a$.  The authors also give a Fourier transform theorem:

\begin{thm}
\cite{benzvi2025potentcategoricalrepresentations} For $V$ a finite dimensional vector space with dual $V^*$, we have
$$2\IndCoh^\pi(V) \simeq 2\IndCoh^\pi(V^*)$$
whose cyclic trace recovers the Fourier transform on $\mathcal D$-modules $\mathcal D(V)\simeq \mathcal D(V^*)$.
\end{thm}

The authors also gave an ansatz:
$$RW(T^*X) = 2\IndCoh^\pi(X),$$
which proposes $2\IndCoh^\pi(X)$ as a model for Rozansky--Witten theory of $T^*X$.  Note that this category does have nonconic objects based on the discussion above.  However although $2\IndCoh$ has very nice functorial properties, objects in the category may be considered too big since there is no restrictions on the singular supports, namely the singular supports are no longer Lagrangians.  For example, infinite union of Lagrangian singular supports can appear.  The authors also extend this definition to open sets $U$ in $T^*X$ following Kashiwara--Schapira's theory of microlocal sheaves \cite{KS} by defining
$$RW(U) = 2\IndCoh^\pi(X) / \{\text{branes supported away from } U\}.$$
Further extending this, the authors consider a Hamiltonian reduction $M$ from an open set $U$, and the functoriality of $2\IndCoh^\pi$ gives a definition of $RW(M)$ in terms of a monad acting on $RW(U)$.  Both our work and \cite{benzvi2025potentcategoricalrepresentations} don't seem to have a straight forward way to go further than this, in the sense that there is no algebraic Darboux theorem as explained in \cite{366422}, so gluing using cotangent covers for a general complex symplectic $Y$ is not viable.  In constrast to \cite{Brav_2015}, where the authors work directly with patches and glue, both our work and \cite{benzvi2025potentcategoricalrepresentations} go from $X$ or rather $X\times\GG_a$ to $T^*X$ and are not working with patches directly.  Ths authors of \cite{D_Agnolo_2007} on the other hand works analytically rather than algebraically.

Our work can be viewed as a sheaf of categories quantization of the complex symplectic manifold $T^*X$, where we construct sheaf of categories with microsupport on complex Lagrangians of the form $df + T_Z^*X$ and subsets of unions of these.  In \cite{D_Agnolo_2007}, for each smooth complex Lagrangian $L$ in a complex symplectic manifold $X$, D'Agnolo and Schapira constructed twisted simple holonomic deformation quantization modules $M_L$ supported on it.  On the other hand, given algebraic Lagrangians $L$, $M$ in $X$, $L\times_X M$ is $-1$-shifted symplectic \cite{ptvv}.  This equips the associated classical scheme with a d-critical locus structure \cite{bbj} in the sense of \cite{joyce_2013} and in \cite{Brav_2015}, Brav et al constructed a perverse sheaf $P_{L,M}$ on the intersection up to canonical isomorphism by gluing sheaves of vanishing cycles in critical charts that exist due to the d-critical locus structure.  The complex analytic case is proven in \cite{bussi2014categorificationlagrangianintersectionscomplex}.  In \cite{gunningham2024deformationquantizationperversesheaves}, Gunningham and Safronov proved that
$$\Hom(M_L,M_M)[\mathrm{dim}X] \simeq P_{L,M}\otimes_\CC \CC((\hbar)),$$
where $P_{L,M}$ is taken to be the perverse sheaf over $\CC$ as mentioned above.  The authors also constructed a holomorphic Fukaya category $\mathbb DRH_X$ with objects $M_L$ for each complex Lagrangian submanifold $L$ and that given a pair of Lagrangians $L,M$ 
$$\Hom_{\mathbb DRH_X}(M_L, M_M)[\mathrm{dim} X/2] \simeq P_{L,M},$$
so the graded Hom space is given by the cohomology of $P_{L,M}$.  Since the periodic cyclic homology of the category of matrix factorizations is the cohomology of vanishing cycles, one can view that the category of matrix factorizations as a categorification of vanishing cycles \cite{efimov}\cite{brtz}.  Some have conjectured a categorification of the result of gluing perverse sheaves of vanishing cycles on a $-1$-shifted symplectic scheme $X$, i.e. the gluing of matrix factorization categories locally given by $\mathrm{MF}(U, W)$ the scheme is locally modelled as the critical locus of the potential $W: U\rightarrow \A^1$  \cite{joyceppt}\cite{bussi2014categorificationlagrangianintersectionscomplex}\cite{toenicm}\cite{toda}\cite{todaz2}, which in particular includes the case where $X$ is an intersection of complex Lagrangians.  The gluing of categories of matrix factorizations would be extremely helpful in the construction of the Kapustin--Rozansky 2-category.  In \cite{hennion2025i}\cite{hennion2025ii}, Hennion et al showed that for singularity invariants on a $-1$-shifted symplectic derived Deligne--Mumford stack $X$ that satisfy a form of isotopy invariance and a quadratic stabilization formula coming from Kn\"{o}rrer periodicity, which include milnor numbers and the perverse sheaf of vanishing cycles, such local invariants can be glued into a global invariant on $X$.  The authors hence recovered the results in \cite{behrend} and \cite{Brav_2015}.  Without establishing isotopy invariance for matrix factorizations, the authors nonetheless proved that the ind-completion of the category of matrix factorizations $\mathrm{IndMF}(U,W)$ can be glued along $X$ as a sheaf of dualizable crystals up to isotopies of functors.  In our work, we bypassed the difficulty of gluing in the case of $T^*X$.

On another direction, in \cite{R1}\cite{R2}, Riva constructed a commutative Rozansky--Witten 3-category with objects being symplectic derived stacks, 1-morphisms being Lagrangian spans, 2-morphisms between Lagrangian spans $X \leftarrow L_i \rightarrow Y$ being the $\ZZ/2\ZZ$-graded quasicoherent sheaves on the derived intersection $L_1\times_{X\times Y} L_2$.  The reason this is called a commutative theory is because if we look at the category of boundary conditions of target $X$, its objects are the Lagrangian spans of the form $\mathrm{pt} \leftarrow L\rightarrow X$, and the Hom between Lagrangians in this model will be $\ZZ/2\ZZ$-graded quasicoherent sheaves on the derived intersection of Lagrangians instead of matrix factorizations.  In \cite{Brunner_2023}\cite{Brunner_2025}, Brunner, Carqueville and Roggenkamp studied truncated affine Rozansky--Witten 3-category whose objects are of the form $T^*\CC^n$, 1-morphisms being graphs of differentials of polynomials and 2-morphisms being matrix factorizations.  The authors proved that none of the objects are fully dualizable since the targets $T^*\CC^n$ are not compact, but after truncating to 2 dimensions, with 2-morphisms being homotopy equivalence classes, the objects are fully dualizable.  So the authors proved that these objects gives rise to fully extended TFTs.  The authors also proved that all of the 1-morphisms have canonically identified left and right adjoints.  In \cite{R2}, Riva proved that there exists a symmetric monoidal 2-functor that maps the truncated affine models to homotopy 2-category of the commutative Rozansky--Witten 3-category which image lands in the subcategory spanned by cotangent stacks of the form $T^*\CC^n$, and that the functor is surjective on those objects.

\subsection{Conventions}

We shall work in the realm of $\infty$-categories in the sense of \cite{L}, where an $(\infty,1)$-category is a weak Kan complex, a simplicial set that satisfies the inner horn condition.  From then on by category we will mean an $(\infty,1)$-category, by an $n$-category we will mean an $(\infty,n)$-category.  If we fix a ring $k$ of characteristic 0, then the $k$-linear stable $(\infty,1)$-categories are equivalent to $k$-linear pre-triangulated dg categories \cite{cohn2016differentialgradedcategoriesklinear}.  From now on, we would treat the two as synonymous.

Given a monoidal category $\mathcal C$ and an algebra object $A$ in $\mathcal C$, we can define the module objects in $\mathcal C$, which are objects $M$ in $\mathcal C$ with unital and associative maps $A\otimes M\rightarrow M$ up to coherent homotopy \cite{HA}.  Such modules form a category we denote as $\mathrm{Mod}_A(\mathcal C)$.  Throughout the paper, we make references of $\Perf(X)$ modules and $\QCoh(X)$ modules and by this we mean $\mathrm{Mod}_{\Perf(X)}(\mathrm{Cat}_\infty^{\mathrm{perf}})$ and $\mathrm{Mod}_{\QCoh(X)}(\mathcal{P}r_{st}^L)$, which we mean by module objects in the category of small stable idempotent complete categories over $\Perf(X)$ and module obejcts in the category of presentable stable categories over $\QCoh(X)$ respectively, where both $\Perf(X)$ and $\QCoh(X)$ has the standard tensor monoidal product.

\subsection{Acknowledgements}
The author would like to thank his advisor, David Nadler, without whom this paper would not be possible.  The author has benefitted immensely from his mathematical vision, expertise, guidance and encouragement throughout this whole journey.  The author would also like to thank Germ\'{a}n Stefanich who generously shared his ideas and insights with me which formed the basis of this paper.  The author also thank David Ben-Zvi, Dave Benson, Andrew J. Blumberg, Dennis Chen, Eita Haibara, Christopher Kuo, Zhouhang Mao, J. Peter May, Aaron Mazel-Gee, Maxime Ramzi, Constantin Teleman and the mathoverflow user Thorgott.  The author was partially supported by Croucher Foundation.

\section{Primer on Derived Geometry}

\subsection{Derived Schemes and Stacks}
Throughout this paper, the spaces that we deal with will mainly be smooth, quasi-compact, Noetherian schemes, but we also deal with the intersection of subschemes.  To obtain the correct intersections, we will need to take the derived fiber product which results in derived schemes.  For example, if we look at the intersection of $\{x=0\}$ and $\{y=0\}$ and the intersection of $\{x=0\}$, $\{y=0\}$ and $\{x=y\}$ both in $\A^2$, scheme theoretically they are equal, but in derived geometry, if we take derived intersections these two are not the same.  We will also make use of integral transforms of sheaves \cite{bzfn2010}\cite{bznp2017} which take place in the derived world.  Even though the notion of derived schemes is enough for our purposes, we will start with derived prestacks, which are the most basic object in derived geometry and start from there.  Derived stacks are just derived prestacks that satisfies descent, or in other words those that form a sheaf, and this notion generalizes classical stacks both in the source and in the target: the source of the functor goes from the category of commutative rings to the category of derived rings, either using simplicial commutative rings or connective commutative differential graded algebras when in characteristic zero; the target of the functor goes from groupoids to $\infty$-groupoids or Kan complexes.  Derived schemes are just derived stacks that has a Zariski atlas.

We will follow closely the point of view and results presented in \cite{gr}.  The advantage of this is that derived stacks and schemes are the derived prestacks that satisfy certain properties instead of having added structure, we can also define geometric objects like quasicoherent sheaves in full generality on prestacks.  An introduction to the viewpoint of derived geometry can be found in \cite{Anel_2021}\cite{Khan2025}\cite{EP}.  We start with the building blocks of derived geometry:

\begin{defn}
The category of derived affine schemes is defined to be the opposite category of connective commutative differential graded algebra over $k$, which is denoted $\dAff$.  A dg algebra $A$ is connective if $A^i = 0$ for $i>0$ up to quasi-isomorphism.
\end{defn}

Now we define the notion of derived prestack, which is the most general form of space one can build using derived affine schemes.

\begin{defn}
The category of derived prestacks is given by the category of functors from the opposite category of $\dAff$ to $\ooGrpd$, the category of $\infty$-groupoids, i.e. a derived prestack is a $\ooGrpd$-valued presheaf on $\dAff$.
\end{defn}

\begin{defn}
Suppose $f:  X\rightarrow  Y$ is a map of derived prestacks.  Then $f$ is affine schematic if for every $S$ in $\dAff_{/ Y}$, the derived prestack $S\times_{ Y}  X$ can be represented by a derived affine scheme.
\end{defn}

With the notion of derived prestack being a presheaf, this suggests that with some topology on $\dAff$ we would be able to single out a class of prestacks that satisfy some descent condition that makes it a sheaf, which we would call derived stacks.  We will first go over how to define a Grothendieck topology on $\dAff$ following \cite{tv2}\cite{gr}.

\begin{defn}
A map $\Spec(B)\rightarrow \Spec(A)$ between derived affine schemes is flat if $H^0(B)$ is a flat $H^0(A)$-module and that one of the following equivalent conditions hold:
\begin{itemize}
\item The natural map
$$H^0(B)\otimes_{H^0(A)} H^*(A) \rightarrow H^*(B)$$
is an isomorphism;
\item For an $A$-module $M$, the natural map
$$H^0(B)\otimes_{H^0(A)} H^*(M) \rightarrow H^*(B\otimes_{A} M)$$
is an isomorphism;
\item Any $A$-module $N$ that is concentrated in degree 0 implies $B\otimes_A N$ is concentrated in degree 0 as well.
\end{itemize}
\end{defn}

\begin{defn}
Let $f: S'\rightarrow S$ be a morphism of derived affine schemes.  $f$ is an \'{e}tale morphism (resp. open immersion) if the following conditions hold:
\begin{enumerate}
\item $f$ is flat and in $\dAff$, the base-change $^{cl}S \times_S S'$ is classical and identifies with $^{cl}S'$;
\item The map of classical affine schemes $^{cl}S'\rightarrow ^{cl}S$ (so in $\mathrm{Aff}$ not in $\dAff$) is an \'{e}tale morphism (resp. open immersion).
\end{enumerate}
\end{defn}

We can likewise define correspondingly for affine schematic morphisms between derived prestacks.

\begin{defn}
Suppose $f:  X \rightarrow  Y$ is an affine schematic morphism between derived prestacks.  $f$ is an \'{e}tale morphism (resp. open immersion) if for every $S$ in $\dAff_{/ Y}$, the corresponding map $S \times_{ Y} X \rightarrow S$ in $\dAff$ is an \'{e}tale morphism (resp. open immersion).
\end{defn}

Now we define an flat (\'etale) topology on $\dAff$:

\begin{defn}
A morphism of derived affine schemes $f: S'\rightarrow S$ is a covering with respect to flat (\'{e}tale) topology if it is \'{e}tale and the induced map on classical affine schemes $^{cl}S'\rightarrow ^{cl}S$ is surjective.
\end{defn}

Now that we have defined what a covering is, we can define descent conditions.

\begin{defn}
A prestack $ Y$ satisfies flat (\'etale) descent if
\begin{itemize}
\item $Y(\varnothing) = \{*\}$;
\item The map $Y(S \sqcup S') \rightarrow Y(S) \times Y(S')$ is an isomorphism;
\item Given $f: S'\rightarrow S$ in $\dAff$ is an flat (\'etale) covering, the corresponding map $ Y(S)\rightarrow \mathrm{Tot}( Y (S'^\bullet/S))$, between $ Y(S)$ and the totalization of $ Y$ evaluating on the \u Cech nerve of $f$, is an isomorphism.
\end{itemize}
\end{defn}

\begin{defn}
A derived prestack $Z$ is a derived stack if it satisfies \'etale descent.
\end{defn}

\begin{lem}
\cite{tv2}\cite{gr} Derived affine schemes are derived stacks.
\end{lem}

Now we can define the notion of derived schemes, namely derived stacks that have a Zariski atlas.  Instead of defining derived schemes as derived prestacks satisfying some conditions, one can define a derived scheme just like how one would define a classical scheme as a locally ringed space, except the structure sheaf $\mathcal O_X$ would now be a connective commutative differential graded algebra over $k$, with it's truncation $H^0(\mathcal O_X)$ giving a classical scheme, and that the sheaves $H^i(\mathcal O_X)$ are quasi-coherent as $H^0(\mathcal O_X)$-modules \cite{toen2014}.

\begin{defn}
\cite{tv2}\cite{gr} A derived prestack $Z$ is a derived scheme if:
\begin{enumerate}
\item $Z$ satisfies \' etale descent;
\item The diagonal morphism $Z\rightarrow Z\times Z$ is affine schematic, and that for any derived affine scheme $T$ over $Z\times Z$, the induced morphism for classical schemes $^{cl} (Z\times_{Z\times Z} T)\rightarrow ^{cl}T$ is a closed immersion;
\item There exists a Zariski atlas, a collection of derived affine schemes $S_i$ over $Z$ such that
\begin{itemize}
\item Each $S_i\rightarrow Z$ is an open immersion;
\item For any derived affine scheme $T$ over $Z$, the induced morphisms for classical schemes $^{cl} (S_i\times_Z T)\rightarrow ^{cl}T$ covers $^{cl}T$.
\end{itemize}
\end{enumerate}
\end{defn}

\begin{defn}
A derived scheme $Z$ is quasi-compact, if $^{cl}Z$ is.
\end{defn}

\subsection{Sheaves on Derived Stacks}
Now we can define quasicoherent sheaves on a derived prestack, continuing to follow the treatment of \cite{gr}.  We first define quasicoherent sheaves on a derived affine scheme.

\begin{defn}
For a derived affine scheme $\Spec A$, we define $\QCoh(\Spec A)$ to be $A$-\emph{mod}.
\end{defn}

When we have a morphism between derived affines $f: \Spec A\rightarrow \Spec B$, this is the same as a map $B\rightarrow A$, and so we obtain the map $f_*: \QCoh(\Spec A)\rightarrow \QCoh(\Spec B)$ by restriction.  The map $f_*$ admits a left adjoint, which is denoted $f^*$.

\begin{defn}
Suppose $ Y$ is a derived prestack.  Then we define
$$\QCoh( Y) = \lim_{S\xrightarrow{u}  Y} \QCoh(S),$$
where the limit is taken over derived affine schemes over $ Y$, with the arrows given by $S\xrightarrow{f} S', \QCoh(S')\xrightarrow{f^*} \QCoh(S)$.
\end{defn}

We can think of an object in $\QCoh( Y)$ to be objects in $\QCoh(S)$ whenever $S$ is a derived affine over $ Y$ satisfying compatibilities with respect to composition of morphisms.  We can as well first build $\QCoh$ as a functor from $\mathrm{dAff^{op}}$ to presentable stable categories and perform right Kan extension along $\mathrm{dAff^{op}}\hookrightarrow \mathrm{PreStk^{op}}$.  Now $\QCoh( Y)$ will just be the value of the functor on $ Y$, and for a map $f:  X \rightarrow  Y$, the corresponding map $\QCoh( Y)\rightarrow \QCoh( X)$ under the $\QCoh$ functor will be denoted as $f^*$.  When viewed this way, the functor $\QCoh$ satisfies flat descent.

\begin{lem}
When $Z$ is a derived scheme, $\QCoh(Z)$ is glued from its open derived affines, i.e. it is sufficient to consider only open immersions $S\rightarrow Z$ in the limit of Definition 2.2.2.
\end{lem}

Given $Y_1, Y_2$ prestacks, we have
\begin{align*}
\QCoh(Y_1\times Y_2) &= \lim_{S\rightarrow Y_1\times Y_2} \QCoh(S) \\
&= \lim_{S_1\rightarrow Y_1, S_2\rightarrow Y_2} \QCoh(S_1\times S_2) \\ 
&= \lim_{S_1\rightarrow Y_1, S_2\rightarrow Y_2} \QCoh(S_1)\otimes \QCoh(S_2),
\end{align*}
where the middle equality comes from the fact that $(S_1\rightarrow Y_1) \times (S_2\rightarrow Y_2)\rightarrow S_1\times S_2 \rightarrow Y_1\times Y_2$ is cofinal.  So we have a canonical map
\begin{align*}
\QCoh(Y_1)\otimes \QCoh(Y_2) &= (\lim_{S_1\rightarrow Y_1}\QCoh(S_1)) \otimes (\lim_{S_2\rightarrow Y_2}\QCoh(S_2)) \\
&\rightarrow \lim_{S_1\rightarrow Y_1, S_2\rightarrow Y_2} \QCoh(S_1)\otimes \QCoh(S_2) \\
&=\QCoh(Y_1\times Y_2)
\end{align*}
which we denote by $\mathcal E, \mathcal F \mapsto \mathcal E \boxtimes \mathcal F$.  The symmetric monoidal product on $\QCoh(Y)$ is then given by
$$\mathcal E \otimes \mathcal F = \Delta^*(\mathcal E \boxtimes \mathcal F),$$
where $\Delta$ is the diagonal map.  Given a morphism $f: X\rightarrow Y$, the functor $f^*$ is then naturally symmetric monoidal.

We now define a full subcategory of perfect complexes.  

\begin{defn}
For $ Y$ a derived prestack, $\Perf( Y)$ is the full subcategory consisting of dualizable objects.
\end{defn}

\begin{lem}
An object in $\QCoh( Y)$ is perfect if and only if for every derived affine $S$ over $ Y$, the corresponding object in $\QCoh(S)$ is perfect.
\end{lem}

For a map $f$ between derived prestacks, the functor $f^*$ sends perfect objects to perfect objects.  So we obtain a functor from $\mathrm{PreStk^{op}}$ to $\ooCat$.  This functor is isomorphic to the right Kan extension of its restriction to $\mathrm{dAff^{op}}$.  This functor, just like before, satisfies flat descent.

We will now go over a broad class of spaces called perfect stacks \cite{bzfn2010} of which quasi-compact derived schemes with affine diagonal and quotients of quasi-projective derived scheme by a linear action of an affine group (in characteristic zero) are examples of.  We can also construct perfect stacks by taking fiber products of perfect stacks and forming mapping stacks between a finite simplicial set and a perfect stack \cite{bzfn2010}.  Roughly speaking, perfect stacks are spaces where one can go back and forth between small and large categories of sheaves, namely we can obtain the category of quasicoherent sheaves from the category of perfect sheaves using the $\mathrm{Ind}$-construction, and we can recover the category of perfect sheaves from the category of quasicoherent sheaves by taking compact objects \cite{bzfn2010}, i.e. for a perfect stack $X$:
$$\QCoh(X)\simeq \mathrm{Ind}\Perf(X), \,\,\,\,\,\,\, \Perf(X) \simeq \QCoh(X)^\omega.$$
We will now give the definition of perfect stacks.

\begin{defn}
\sloppy
\cite{bzfn2010} A derived stack $X$ is perfect if it has affine diagonal and that $\QCoh(X) \simeq \mathrm{Ind}\Perf(X)$.  A morphism $X\rightarrow Y$ is perfect if for derived affines $U$ over $Y$, $X \times_Y U$ is perfect.
\end{defn}

\begin{lem}
\cite{bzfn2010} On a perfect stack $X$, compact objects of $\QCoh(X)$ are the same thing as perfect complexes.
\end{lem}

\subsection{Sheaves of Categories}
Throughout this paper, we view sheaves of categories from the point of view of modules over $\QCoh(X)$ and $\Perf(X)$.  Now we have to specify what environment we are working in, i.e. where the module objects live in.  We will work mainly with the category of small stable idempotent complete categories $\mathrm{Cat}_\infty^{\mathrm{perf}}$ \cite{Blumberg_2013}.  The category $\mathrm{Cat}_\infty^{\mathrm{perf}}$ is equivalent to the category of stable compactly generated categories $\mathcal Pr_{st, \omega}^L$, with colimit preserving and compact obejct preserving functors.  This equivalence is given by the $\mathrm{Ind}$ construction \cite[\nopp 5.5.7.10]{L}.  The category of presentable categories $\mathcal Pr^L$ with colimit preserving functors has a natural closed symmetric monoidal structure \cite[\nopp 4.8]{HA}, where functors
$$\mathcal C\otimes \mathcal D \rightarrow \mathcal E$$
is equivalent to functors
$$\mathcal C\times \mathcal D\rightarrow \mathcal E$$
which are colimit preserving at both variables, and this monoidal product restricts to the subcategory $\mathcal Pr_{st,\omega}^L$ and the full subcategory of stable presentable categories $\mathcal Pr_{st}^L$.  This means that we can define a closed monoidal product on $\mathrm{Cat}_\infty^{\mathrm{perf}}$ \cite{bzfn2010} by
$$\mathcal C\otimes \mathcal D = (\mathrm{Ind}(\mathcal C)\otimes \mathrm{Ind}(\mathcal D))^\omega.$$
$\QCoh(X)$ is an algebra object in the category of stable presentable categories $\mathcal Pr_{st}^L$, with colimit preserving functors and $\Perf(X)$ is an algebra object in the category of small stable idempotent complete categories $\mathrm{Cat}_\infty^{\mathrm{perf}}$.  Given an algebra object $A$ in a monoidal category $\mathcal C$, the module objects $M$ in $\mathcal C$ over $A$ are the ones with unital and associative maps $A\otimes M\rightarrow M$ up to coherent homotopy \cite{lurie2007derivedalgebraicgeometryii}.  We can then define tensor product $\cdot \otimes_A \cdot$ between modules over $A$ by the bar construction \cite[\nopp 4.4]{HA}.  So when we talk about $\QCoh(X)$ modules, we mean module objects in $\mathcal Pr_{st}^L$ over the algebra object $\QCoh(X)$ and by $\Perf(X)$ modules, we mean module objects in $\mathrm{Cat}_\infty^{\mathrm{perf}}$ over the algebra object $\Perf(X)$.  

Now we will go through two useful theorems regarding integral transforms of sheaves that will be used in calculations later.  The first theorem more or less states that the category of quasicoherent kernels is equivalent to colimit preserving functors between two categories of quasicoherent sheaves as $\QCoh(S)$ modules, and that the category of kernels is equivalent to the tensor product of the two $\QCoh(S)$ modules \cite{bzfn2010}.  The second theorem gives an equivalence between the category of coherent kernels and exact functors between the category of perfect sheaves and the category of coherent sheaves as $\Perf(S)$ modules \cite{bznp2017}.

\begin{thm}
\cite{bzfn2010}
\begin{enumerate}
\item For $X\rightarrow S \leftarrow Y$ morphisms of perfect stacks, there is a canonical equivalence
$$\QCoh(X\times_S Y) \simeq \QCoh(X)\otimes_{\QCoh(S)}\QCoh(Y).$$
\item For $X\rightarrow S$ a perfect morphism to a derived stack $S$ with affine diagonal, i.e. the fibers over affines is perfect, and $Y\rightarrow S$, there is a canonical equivalence
$$\Phi: \QCoh(X\times_S Y) \xrightarrow{\ \ \sim\ \ } \mathrm{Fun}_{\QCoh(S)}^L(\QCoh(X),\QCoh(Y))$$
where the right hand side are the colimit preserving $\QCoh(S)$-linear functors between $\QCoh(X)$ and $\QCoh(Y)$, given by
$$\Phi(K) = p_{Y*}(p_X^*(-)\otimes K).$$
Now suppose that $X,Y,Z$ over $S$ all perfect stacks, we have composition of functors
$$\mathrm{Fun}_{\QCoh(S)}^L(\QCoh(X),\QCoh(Y))\otimes \mathrm{Fun}_{\QCoh(S)}^L(\QCoh(Y),\QCoh(Z))$$
$$\rightarrow \mathrm{Fun}_{\QCoh(S)}^L(\QCoh(X),\QCoh(Z))$$
which corresponds to the convolution map
$$\QCoh(X\times_S Y)\otimes \QCoh(Y\times_S Z)\rightarrow \QCoh(X\times_S Z)$$
given by
$M\star N = p_{13*}(p_{12}^*(M)\otimes p_{13}^*(N))$ where the maps $p_{12},p_{23},p_{13}$ comes from the diagram
\begin{center}
\begin{tikzcd}
            & X\times_S Y\times_S Z \arrow[ld, "p_{12}"'] \arrow[d, "p_{13}"'] \arrow[rd, "p_{23}"] &             \\
X\times_S Y & X\times_S Z                                                                           & Y\times_S Z.
\end{tikzcd}
\end{center}

\end{enumerate}
\end{thm}

\begin{rmk}
Note that for Theorem 2.20 (2), if $X$ and $S$ are both perfect stacks, then $X\rightarrow S$ is immediately a perfect morphism since fiber products of perfect stacks are perfect \cite{bzfn2010}.
\end{rmk}

\begin{thm}
\cite{bznp2017} 
\begin{enumerate}
\item For $X\rightarrow S \leftarrow Y$ morphisms of perfect stacks, the equivalence in Theorem 2.20 restricts to an equivalence on compact objects
$$\Perf(X\times_S Y)\simeq \Perf(X)\otimes_{\Perf(S)}\Perf(Y).$$
\item Suppose $S$ is a perfect stack, $p_X: X\rightarrow S$ a proper relative algebraic space, $Y$ a locally Noetherian $S$-stack.  The integral transform construction provides an equivalence
$$\Phi: \Coh(X\times_S Y)\xrightarrow{\ \ \sim\ \ }\mathrm{Fun}_{\Perf(S)}^{ex}(\Perf(X), \Coh(Y)),$$
where $\Coh(X) \subset \QCoh(X)$ is the full subcategory with objects of bounded and coherent cohomologies.
\end{enumerate}
\end{thm}

When $Y$ is regular, $\Perf(Y) = \Coh(Y)$ and so functors between perfect sheaves are given by coherent kernels, and we would have a similar correspondence between convolution on the left and the composition of functors on the right \cite{bznp2017}.

Now we are equipped to talk about sheaves of categories as defined and studied in \cite{G2015}.  The general objects of Rozansky--Witten theory of a complex symplectic manifold are sheaves of categories over complex Lagrangians \cite{KAPUSTIN2009295}\cite{KR}.  Even though in the later sections we are modelling such a theory using $\Perf(X)$ modules, here we give an introduction of quasicoherent sheaves of categories and establish an equivalence with $\QCoh(X)$ modules for certain nice spaces $X$.  It is with such an intuition in mind, that we view $\Perf(X)$ modules as sheaves of categories.  We shall go over the definitions and some results of \cite{G2015}.

\begin{defn}
We first define a functor from $\dAff^{\mathrm{op}}$ to $\ooCat$ given by
$$S \mapsto \mathrm{Mod}_{\QCoh(S)}(\mathcal{P}r_{st}^L),$$
and perform right Kan extension along $\dAff^{\mathrm{op}}\hookrightarrow \mathrm{PreStk}^{\mathrm{op}}$ to obtain the functor $\ShvCat$.  The value of $\ShvCat$ on prestack $ Y$ is the quasi-coherent sheaves of categories on $ Y$.
\end{defn}

The category $\ShvCat( Y)$ is the limit of $\mathrm{Mod}_{\QCoh(S)}(\mathcal{P}r_{st}^L)$ where the limit is taken over derived affine schemes over $ Y$ with the arrows given by
$$S\xrightarrow{f} S', \quad \mathrm{Mod}_{\QCoh(S')}(\mathcal{P}r_{st}^L)\xrightarrow{\QCoh(S)\otimes_{\QCoh(S')} (-)} \mathrm{Mod}_{\QCoh(S)}(\mathcal{P}r_{st}^L).$$
We can think of an object in $\ShvCat( Y)$ to be objects in $\mathrm{Mod}_{\QCoh(S)}(\mathcal{P}r_{st}^L)$ whenever $S$ is a derived affine over $  Y$ satisfying compatibilities with respect to composition of morphisms.

\begin{lem}
The category $\ShvCat( Y)$ has a symmetric monoidal structure, given by component-wise tensor product, with unit $\QCoh_{/ Y}$ which assigns $\QCoh(S)$ for each derived affine $S$ over $Y$.
\end{lem}

Given an object $\mathcal F$ in $\ShvCat( Y)$, we have a functor $\Gamma(-,\mathcal F)$ that gives a stable presentable category for each derived affine $S$ over $ Y$.  We can perform right Kan extension along $\dAff_{/ Y}^{\mathrm{op}}\hookrightarrow \mathrm{PreStk}_{/ Y}^{\mathrm{op}}$, obtain a section functor that gives sections of $\mathcal F$ over any prestack $ X$ over $ Y$.  The sections $\Gamma( X,\mathcal F)$ is given by the limit of $\Gamma(S,\mathcal F)$ with derived affines $S$ over $ X$ with the arrows given just as above.

\begin{ex}
For $ X$ over $ Y$, $\Gamma( X, \QCoh_{/ Y}) = \QCoh( X)$.
\end{ex}

So now we obtained a functor $ X \mapsto \Gamma( X, \mathcal F)$.

\begin{lem}
For $ X$ over $ Y$, the functor $\Gamma( X, -)$ with source in $\ShvCat( Y)$ is lax symmetric monoidal.
\end{lem}

Since $\QCoh_{/ Y}$ is the unit in $\ShvCat( Y)$, there is a $\Gamma( X, \QCoh_{/ Y}) = \QCoh( X)$ action on $\Gamma( X, \mathcal F)$, $\Gamma( X, \mathcal F)$ is a $\QCoh(X)$ module.  So from now on we consider the functor $\Gamma( X, -)$ to be from $\ShvCat( Y)$ to $\mathrm{Mod}_{\QCoh(X)}(\mathcal{P}r_{st}^L)$.  In particular, we have the functor $\Gamma_{ Y} = \Gamma( Y, -): \ShvCat( Y)\rightarrow \mathrm{Mod}_{\QCoh(Y)}(\mathcal{P}r_{st}^L)$.

\begin{lem}
The functor $\Gamma_{ Y}$ has a left adjoint $\mathrm{Loc}_ Y$, which is given by:
$$\Gamma(S, \mathrm{Loc}_ Y (M)) = \QCoh(S)\otimes_{\QCoh( Y)} M,$$
where $M$ is a $\QCoh( Y)$ module and $S$ a derived affine over $ Y$.
\end{lem}

\begin{defn}
A derived prestack $ Y$ is 1-affine if the functors $\Gamma_ Y$ and $\mathrm{Loc}_ Y$ are mutually inverse equivalences.
\end{defn}

\begin{thm}
\cite{G2015} Any quasi-compact, quasi-separated derived scheme is 1-affine.
\end{thm}

We now define some functors between $\QCoh(X)$ and $\QCoh(Y)$ modules given a morphism $f:X\rightarrow Y$ following \cite{G2015}.  Suppose we have a morphism $f:X\rightarrow Y$, then we have $f^*: \QCoh(Y)\rightarrow \QCoh(X)$ which is symmetric monoidal.  We have a natural functor $\res_f: \mathrm{Mod}_{\QCoh(X)}(\mathcal{P}r_{st}^L)\rightarrow \mathrm{Mod}_{\QCoh(Y)}(\mathcal{P}r_{st}^L)$ by restricting the action
$$\QCoh(Y)\otimes M \xrightarrow{f^*\otimes \mathrm{Id}} \QCoh(X)\otimes M\rightarrow M.$$
We also have the induction functor $\ind_f: \mathrm{Mod}_{\QCoh(Y)}(\mathcal{P}r_{st}^L)\rightarrow\mathrm{Mod}_{\QCoh(X)}(\mathcal{P}r_{st}^L)$ given by
$$M \mapsto \QCoh(X)\otimes_{\QCoh(Y)} M$$
which is a left adjoint to $\res_f$.  Now Suppose we have $p_1:X\times Y\rightarrow X$, $p_2: X\times Y\rightarrow Y$ projection maps, then the exterior tensor of a $\QCoh(X)$ module $M$ and a $\QCoh(Y)$ module $N$ is defined to be
$$M\boxtimes N = \ind_{p_1}M\otimes_{\QCoh(X\times Y)} \ind_{p_2}N.$$

However, throughout this paper we will work with $\Perf(X)$ modules.  $\Perf(X)$ inherits the symmetric monoidal structure and the functors $f^*$ preserves compact objects so will map perfect sheaves to perfect sheaves \cite{gr}.  The functors and operations above can likewise be defined on $\Perf(X)$ modules.  In light of Theorem 2.3.10, we should consider $\Perf(X)$ modules to be sheaves of categories on $X$ with some finiteness conditions.  The category $\mathrm{Cat}_\infty^{\mathrm{perf}}$ is compactly generated, complete and cocomplete \cite[Corollary 4.25]{Blumberg_2013}, in particular $\mathrm{Cat}_\infty^{\mathrm{perf}}$ is presentable.  So the category of module objects $\perfxmod$ is presentable \cite[Corollary 4.2.3.7]{HA} and compactly generated \cite[Lemma 2.1]{aoki2025higherpresentablecategorieslimits}.

\section{Matrix Factorizations}
There are two categories we can associate to a potential $W: X\rightarrow \A^1$.  One is the category of singularities $\mathrm{DSing}(X_s) = \Coh(X_s)/\Perf(X_s)$ \cite{Orlov_2003}, where $X_s = X\times_{\A^1}s$ is the fiber at $s$ and that the Verdier quotient of categories is taken in a suitable category of categories.  For such a definition to make sense, we would need perfect complexes to be coherent, and this would require our spaces to be eventually coconnective, i.e. our spaces are covered by affines $\Spec(A)$ where $H^{-i}(A)=0$ for $i>>0$, which is equivalent to the structure sheaf being coherent \cite{indcoh}.  When $X$ is regular, $\Coh(X) = \Perf(X)$, and so the category of singularities $\mathrm{DSing}(X)=0$, hence the name.  This category measures the singularities of $X$, and localizes to the singularities of $X$: suppose that an open subscheme $i: U\rightarrow X$ contains the singularities of $X$, then the map induced by the pullback $i^*$ on $\Coh(X)$ gives an equivalence on $\mathrm{DSing}(X)$ and $\mathrm{DSing}(U)$ \cite[Proposition 1.14]{Orlov_2003}.  The other one is called the category of matrix factorizations which is introduced in \cite{EMF}.  This category is a $\ZZ/2\ZZ$-graded category of curved complexes of coherent sheaves, i.e. the differentials does not square to zero but is the potential $W$ multiplied by the identity map.  For simplicity, we go through the definition of matrix factorizations assuming $X = \Spec(A)$ is affine.  The non-affine case is treated in \cite{Orlov_2011}.  The canonical reference is \cite{Orlov_2003}.

\subsection{Category of Singularities}
\begin{defn}
The category of singularities of a eventually coconnective Noetherian derived scheme $X$ $\mathrm{DSing}(X)$ is given to be the cofiber
\begin{center}
\begin{tikzcd}
\Perf(X) \arrow[d] \arrow[r, hook] & \Coh(X) \\
0                                  &        
\end{tikzcd}
\end{center}
in $\mathrm{Cat}_\infty^{\mathrm{perf}}$.
\end{defn}

Here we have taken our quotient in $\mathrm{Cat}_\infty^{\mathrm{perf}}$.  There are two environment to take this quotient in: one is our $\mathrm{Cat}_\infty^{\mathrm{perf}}$, the category of small stable idempotent complete categories which both $\Coh(X)$ and $\Perf(X)$ belongs and the other is the category of small stable categories with exact functors as morphisms $\mathrm{Cat}_\infty^{\mathrm{ex}}$.  Both $\mathrm{Cat}_\infty^{\mathrm{ex}}$ and $\mathrm{Cat}_\infty^{\mathrm{perf}}$ are cocomplete \cite[Corollary 4.25]{Blumberg_2013}.  Since the idempotent completion functor
$$\mathrm{Idem}: \mathrm{Cat}_\infty^{\mathrm{ex}} \rightarrow \mathrm{Cat}_\infty^{\mathrm{perf}}$$
is left adjoint to the inclusion $\mathrm{Cat}_\infty^{\mathrm{perf}} \hookrightarrow \mathrm{Cat}_\infty^{\mathrm{ex}}$ \cite{Blumberg_2013}, the idempotent completion functor preserves colimits, so we see that the quotient taken in $\mathrm{Cat}_\infty^{\mathrm{perf}}$ is the idempotent completion of the quotient taken in $\mathrm{Cat}_\infty^{\mathrm{ex}}$.  So the homotopy category of the quotient in $\mathrm{Cat}_\infty^{\mathrm{perf}}$, is the idempotent completion of $\mathrm{Ho}(\Coh(X))/\mathrm{Ho}(\Perf(X))$ \cite[Proposition 5.15]{Blumberg_2013}.  So our quotient is a stable $(\infty,1)$ enhancement of the idempotent completed traingulated category of singularities of Orlov \cite{Orlov_2003}.  Our definition for a fiber at 0 $\mathrm{DSing}(X_0)$ is then equivalent to the definition of Preygel's category of matrix factorizations \cite[Proposition 3.4.1]{preygel2011thomsebastianidualitymatrix}, which is defined by
$$\Coh(X_0)\otimes_{\Perf(k[\beta])} \Perf(k(\beta)),$$
where its tensor product is defined in $\mathrm{Cat}_\infty^{\mathrm{perf}}$ and $\beta$ is of degree 2.  To see why $\Coh(X_0)$ is $k[\beta]$-linear, we see that the derived base loop space of $\A^1$ at $0$ $\mathcal L_0\A^1 = 0\times_{\A^1} 0 = \Spec (k[\eta]), \,\,\,\, |\eta|=-1$ acts by loop composition on the right on $X_0 = X\times_{\A^1} 0$ \cite[Construction 3.1.5]{preygel2011thomsebastianidualitymatrix}, so $\Coh(\mathcal L_0 \A^1) = \Coh(k[\eta])$ acts on $\Coh(X_0)$.  Now by Koszul duality $\Coh(k[\eta]) = \Perf(k[\beta])$.  Preygel's and our definitions are universal in the following sense \cite[Corollary 3.4.3]{preygel2011thomsebastianidualitymatrix}: the exact sequence in $\mathrm{Cat}_\infty^{\mathrm{perf}}$
$$\Perf(X_0)\rightarrow \Coh(X_0) \rightarrow \mathrm{DSing}(X_0)$$
is obtained from the universal example $X = \mathrm{pt}$:
$$\Perf(\mathcal L_0\A^1) \rightarrow \Coh(\mathcal L_0\A^1)\rightarrow \mathrm{DSing}(\mathcal L_0\A^1)$$
by tensoring $\Coh(X_0)\otimes_{\Coh(\mathcal L_0\A^1)}$.

\subsection{Category of Matrix Factorizations}
\begin{defn}
\cite{KD} For $s\in \A^1$, the category of matrix factorizations at $s$ $\mathrm{MF}_s(X,W)$ is a $\ZZ/2\ZZ$-graded category with objects 
$$P = ( \xymatrix{P_0 \ar@/^/[r]^{p_0} & P_1 \ar@/^/[l]^{p_1}} ),$$
with both $P_0$, $P_1$ finitely generated projective $A$-modules and that both $p_0p_1$ and $p_1p_0$ is $(W - s)\cdot \operatorname{id}$.  The morphisms are given by
$$\Hom(P,Q) = \bigoplus_{i,j} \Hom(P_i, Q_j),$$
with grading given by $i-j \mod 2$, with differential given by
$$f \mapsto q f - (-1)^{|f|} fp.$$
\end{defn}

This definition is the dg-enhancement of the one given in \cite{Orlov_2003}.  We can extend this definition to a commutative differential graded algebra $(A,d_A)$ where $W$ is in the even part of $A$ and that $dW = 0$ \cite{block}\cite{KR}.  This is called a curved cdga.  A $\ZZ/2\ZZ$-dg module over $A$ is a graded module $(M,d_M)$ such that $d_M$ and $d_A$ are compatible:
$$d_M(am) = d_A(a) m + (-1)^{|a|}ad_M(m),$$
such that $d_M^2 = (W-s) \cdot \operatorname{id}$.  A $\ZZ/2\ZZ$-dg module is perfect if it is of the form $P\otimes_{A^0} A$ where $P$ is a projective $\ZZ/2\ZZ$-graded module over the even part of $A$, $A^0$ \cite{KR}.  So the resulting category consists of objects that are perfect $\ZZ/2\ZZ$-dg modules over $A$.  Definition 3.2.1 is the special case where $A$ is concentrated in degree $0$.

Given an object $P$ in $\mathrm{MF}_0(X,W)$ and that $W$ is not zero on any component, then we have a short exact sequence \cite{Orlov_2003}
$$0\rightarrow P_1\xrightarrow{p_1} P_0 \rightarrow \mathrm{coker}p_1\rightarrow 0.$$
The sheaf $\mathrm{coker}p_1$ is annhilated by $W$ and so is a sheaf on $X_0$ in $\Coh(X_0)$, and hence we obtain a map from $\mathrm{MF}_0(X,W) \rightarrow \mathrm{DSing}(X_0)$.  This map is an equivalence after idempotent completion \cite[Section 3.2]{Orlov_2003}.  When $W=0$, then by direct inspection, since the category of matrix factorizations is just the $\ZZ/2\ZZ$-graded perfect sheaves, one can see that the two categories are equivalent \cite[\nopp 3.4.4]{preygel2011thomsebastianidualitymatrix}.

Now suppose we are given two complex Lagrangians, represented by the graphs of $df$ and $dg$.  Then in Rozansky--Witten theory the $\Hom$ between these two complex Lagrangians should be given by the category of matrix factorizations of $g-f$, which is supported in the critical points of $g-f$ \cite{KR}.  Therefore what we want would be a direct sum of matrix factorizations of $g-f-s$ for $s\in\A^1$, or else we would be adding an additional constraint and looking at only the critical points in $g-f=0$.  As Kapustin and Rozansky have noted in \cite{KR}, $df$ is defined only up to a locally constant function, or when there is only one component, up to a number.  So we define
$$\mathrm{MF}(X,W) = \bigoplus_{s\in\A^1}\mathrm{DSing}(X_s).$$
Kapustin and Rozansky call this the augmented category.  Note that this definition is different (other than idempotent completion) from what Orlov, following Kontsevich's proposal, calls the category of B-branes in \cite{Orlov_2003}.  The category of B-branes is the direct product of $\mathrm{DSing}(X_s)$ rather than a direct sum.  When $X$ is regular, the product is finite and so the definitions agree.

\section{Rozansky--Witten Theory}

\subsection{Kapustin--Rozansky--Saulina's 2-category}
We will now go over some background on Rozansky--Witten theory and results contained in \cite{KR}. We are interested in the category of boundary conditions of a 3-dimensional topological sigma model.  A topological sigma model of dimension 3 with complex symplectic target gives a 3-category $\mathcal C$ with objects being the target manifolds and a contravariant duality functor that maps $(X,\omega)$ to $(X,-\omega)$, where $\omega$ is the holomorphic symplectic form.  The duality functor has the property that
$$\Hom_\mathcal C(X_1,X_2) = \Hom_\mathcal C(\operatorname{pt},X_1^\vee\times X_2),$$
where $X^\vee$ is the dual of $X$.  So we see that if we understand all the 2-categories of the form
$$\mathcal C_X = \Hom_\mathcal C(\operatorname{pt},X),$$
which is called the boundary conditions of $X$, and how objects of $\mathcal C_{X_1^\vee \times X_2}$ give rise to functors between $\mathcal C_{X_1}$ and $\mathcal C_{X_2}$, we will understand the 3-category $\mathcal C$.

In \cite{KAPUSTIN2009295} and \cite{KR}, Kapustin and Rozansky, along with Saulina in the former work, studied the boundary conditions of Rozansky--Witten theory and that they correspond to complex Lagrangians in the target space $X$ with complex fibrations.  In \cite{KR}, Kapustin and Rozansky modeled the theory with cotangent target using matrix factorizations.  The authors also mentioned how the objects are naturally associated to sheaves of categories.  For simplicity sake we will present here their approach for $T^*\CC^n$:
$$\mathrm{Obj}(\mathrm{KR}(T^*\CC^n)) = \{f \in \CC[x_1,\dots,x_n]\},$$
$$\mathrm{KR}(T^*\CC^n)(f,g) = \mathrm{MF}(\CC^n, g-f).$$
The composition of morphisms is given by
$$\mathrm{MF}(\CC^n, g-f) \times \mathrm{MF}(\CC^n, h-g) \rightarrow \mathrm{MF}(\CC^n, h-f),$$
$$(P,Q)\rightarrow P\otimes_{\CC[x_1,\dots,x_n]} Q, \,\,\,\, d_{P\otimes Q} = d_P\otimes \operatorname{id} + (-1)^{|\cdot|} \operatorname{id}\otimes d_Q.$$
Here the objects are the Lagrangians of $T^*\CC^n$ of the form of a graph $\Gamma_{df}$.  We can extend this definition as in \cite{KR} to fibrations over such Lagrangians by expanding our objects:
$$\mathrm{Obj}(\mathrm{KR}(T^*\CC^n)) = \{((y_1,\cdots,y_k), f): f \in \CC[x_1,\dots,x_n,y_1,\cdots,y_k]\},$$
$$\mathrm{KR}(T^*\CC^n)((y_1,\cdots,y_k, f),(z_1,\cdots,z_l,g)) = \mathrm{MF}(\CC^{n+k+l}, g-f).$$
The composition of morphisms is given similarly by tensor product except now it will not be over $\CC[x_1,\cdots,x_n]$ but over $\CC[x_1,\cdots,x_n,z_1,\cdots,z_l]$, where $z_i$'s are the variables of the intermediate object.  We should think of the objects $(y_1,\cdots,y_k,f)$ as follows \cite{KR}.  First we define
$$CV_f = \{p\in \CC^{n+k}: df_p = 0 \text{ on } T_p\CC^k\},$$
i.e. $df_p$ can be considered to be an element $T^*_p\CC^n$.  So we have a map
$$CV_f \rightarrow T^*\CC^n, \,\,\,\, p \mapsto df_p.$$
Under some nice conditions on $f$, the image of this map gives a Lagrangian submanifold, and that this map is a fibration.  So these objects should be thought of as fibrations over Lagrangian submanifolds.  When $k=0$, what we have is a one point fibration over $df$, and we recover graphs as above.  These, however, are not the most general objects that can be defined in $\mathrm{KR}(T^*\CC^n)$, notice that we only considered functions $f$ defined on vector bundles over $\CC^n$, to get more general objects we look at functions defined on fibrations over $\CC^n$ \cite{KR}.    

\subsection{Dualizability in Affine models}
In \cite{Brunner_2023}, Brunner, Carqueville and Roggenkamp form the affine Rozansky--Witten models 3-category $RW^{aff}$ with objects being $T^*\CC^n$, and that morphisms between $T^*\CC^k$ and $T^*\CC^l$ is given by $\mathrm{KR}(T^*\CC^{k+l})$.  According to the cobordism hypothesis \cite{BD}\cite{LCH}, a fully framed extended topological field theory is determined by a fully dualizable object in the target category (along with $G$-symmetry for the case of considering bordisms with $G$-structure), i.e. such a functor is determined by its value at a point.  The authors of \cite{Brunner_2023} proved that no object except the point in $RW^{aff}$ is fully dualizable, due to the fact that $T^*\CC^n$ is not compact for $n>0$, so one cannot construct a fully extended 3-dimensional topological field theory from the 3-category $RW^{aff}$.  However, the authors proved that if one truncates $RW^{aff}$ to a 2-category, by making the new 2-morphisms to be the isomorphism classes of 2-morphisms, every object in the truncated category is fully dualizable, and there is one fully extended oriented topological field theory for each positive $n$.

\subsection{Relationship with Topological Gauge Theories}
One of the guiding motivations of the author is to understand the following statement in \cite{T}:

\begin{repstat}{teleman}
Pure topological gauge theory in 3 dimensions for a compact Lie group $G$ is equivalent to the Rozansky-Witten theory for the BFM space of the Langlands dual Lie group $G^\vee$.
\end{repstat}

Here we would first have to explain some of the terms in the statement.  By a pure topological gauge theory in 3 dimensions, without delving too deep into topological field theories, we mean the category of $G$-categories, i.e. categories equipped with a topological or locally constant group action \cite{T}.  Intuitively, an action of $G$ on a category $\mathcal{C}$ will require $G$ to permute objects of $\mathcal{C}$ and sending morphisms from $x$ to $y$ to morphisms from $gx$ to $gy$, i.e. a map from $G$ to $\mathrm{End}(\mathcal{C})$.  For the action to be locally constant, a trivial neighbourhood of $e\in G$ acts trivially up to coherent homotopies and so we would have a morphism:
$$\Omega G\rightarrow \mathrm{Aut}(\mathrm{Id}_\mathcal{C}),$$
where $\Omega G$ is the based loop group \cite{T}.  Note that in \cite{T}, Teleman provisionally settles for the linearized definition when working with linear dg-categories, and topological actions of a connected $G$ is given by a $E_2$ algebra map from the singular chains $C_*\Omega G$ with the Pontrjagin product to the Hochschild cochains of $\mathcal C$.

Now suppose we take a step back, and take a look at the traditional theory of group actions on spaces.  A topological group action of a topological space $X$ by $G$  is a group action such that $G \times X\rightarrow X$ is continuous.  This is equivalent to looking at $(\infty,1)$-functors from $BG$ to the category of topological spaces $\mathrm{Top}$ \cite{AD2}.  So in light of this we can define analogously that a $G$-category to be an $(\infty,1)$-functor from $BG$ to the category of $(\infty,1)$-categories.  Here we are being vague and have not specified what environment we work on: stable, presentable etc.  Just as in usual equivariant homotopy theory, Elmendorf's theorem \cite{E}\cite[Chapter VI]{EHCT}\cite{RJP}\cite{MS2}\cite{AD2} tells us that $(\infty, 1)$-functors from $BG$ to $\mathrm{Top}$ is not enough, the above approach is not enough.  We will discuss the issue of what is needed in genuine homotopy theory of $G$-categories in Section 5.2.

Now we look at the other side of the correspondence.  The Langlands dual Lie group $G^\vee$ is the compact Lie group that has the dual root datum of $G$ \cite{AD3}.  The BFM space of a compact connected Lie group is studied by Bezrukavnikov, Finkelberg and Mirkovi\'{c} in \cite{Bezrukavnikov_Finkelberg_Mirkovic_2005}, which is the spectrum of $H_*^{G^\vee}(\Omega G^\vee; \CC)$ with Pontrjagin multiplication as the ring structure \cite{T} which was identified in \cite{Bezrukavnikov_Finkelberg_Mirkovic_2005} with the quotient $C_{G_\CC,\mathfrak g_\CC} / G_\CC$, where
$$C_{G_\CC, \mathfrak g_\CC} = \{(g,x): Ad_g(x) = x, x \text{ is regular}\} \subset G_\CC \times \mathfrak g_\CC,$$
a locally closed subvariety quotient out by the diagonal adjoint action of $G_\CC$ where $G_{\CC}$ is the complexification of $G$.  Under the isomorphism $\mathfrak g_\CC^\vee \simeq \mathfrak g_\CC$, we can dually consider the corresponding subset
$$C_{G_\CC, \mathfrak g_\CC^\vee} = \{(g,x): Ad_g^*(\eta) = \eta, \eta \text{ is regular}\} \subset G_\CC \times \mathfrak g_\CC^\vee,$$
which is a locally closed subvariety of the cotangent bundle of $G_\CC$ and the corresponding quotient given again by the diagonal action.  This quotient can also be viewed as a Hamiltonian quotient from an open subset of $T^*G_\CC$ where the zero fiber of the moment map is given by the corresponding dual of $C_{G_\CC, \mathfrak g_\CC^\vee}$ \cite{T}.  Technically in  \cite{Bezrukavnikov_Finkelberg_Mirkovic_2005}, Bezrukavnikov, Finkelberg and Mirkovi\'{c} are working with almost simple complex algebraic groups $G$, and are studying the equivariant homology of the affine Grassmannian of $G$ over $G(\CC[[t]])$.  However we can look at the maximal compact subgroup $K$ and $\Omega K$ is a topological model for the affine Grassmannian \cite{PS}\cite{Lam_2010}.  There are other characterizations listed in \cite{T}.  One can see that in the case of $G = T$ abelian, $BFM(T) = T^*T_{\CC}$, is just the cotangent bundle of the complexified torus \cite{T}.  On $BFM(G)$ there exist a Lagrangian $Z$ which is the zero fiber of the map inherited from
$$\pi: G_\CC \times \mathfrak g_\CC^\vee \rightarrow \mathfrak g_\CC^\vee / G_\CC,$$
the quotient projection map to the space of coadjoint orbits \cite{T}.  This Lagrangian $Z$ is $\Spec H_*(\Omega G^\vee)$ and when $G = T$ is abelian, the Lagrangian $Z$ is the zero section of $T^*T_\CC$ \cite{T}.

By Rozansky--Witten theory of $X$ we mean the category of boundary conditions of $X$ as mentioned in Section 4.1.  The more general objects one can have in this category would be sheaves of categories on the complex symplectic manifold $X$ with support on the complex Lagrangians of $X$ \cite{KAPUSTIN2009295}\cite{KR}\cite{T}.  Note that we have once again been vague and have not specified what types of sheaves of categories we are working with.  So the correspondence relates the category of $G$-categories with the $2$-category of boundary conditions of the complex symplectic manifold $BFM(G^\vee)$, which is the category of sheaves of categories with support on complex Lagrangians of $BFM(G^\vee)$.  Under this picture, we can view $G$-categories as sheaves of categories with Lagrangian support on $BFM(G^\vee)$ and that the intersections corresponds to the $\Hom$ of $G$-categories \cite{T}.  Now on the side of $G$-categories we have the trivial and regular representations just like in normal representation theory, where $\Hom(\mathrm{Triv}, \mathcal V) = \mathcal V^G$ is the fixed point category and $\Hom(\mathrm{Reg}, \mathcal V) = \mathcal V$ is the underlying category.  So correspondingly on the side of Rozansky--Witten theory, we have corresponding Lagrangians or sheaves of categories with Lagrangian support $Z_{\mathrm{triv}}, Z_{\mathrm{reg}}$, such that $\Hom(Z_{\mathrm{triv}}, \mathcal C)$ is the corresponding object in Rozansky--Witten theory of the fixed point category and $\Hom(Z_{\mathrm{reg}}, \mathcal C) = \mathcal C$ \cite{T}.  $Z_{\mathrm{reg}}$ is the Lagrangian $Z = \Spec H_*(\Omega G)$ in $BFM(G^\vee)$ as mentioned above and corresponds to the zero section in $T^*T_\CC^\vee$ in the case of abelian $G = T$.  When $G = T$ abelian, we additionally know that $Z_{\mathrm{triv}}$ is the conormal $T_{1}^*T_\CC^\vee$ at $1$ \cite{T}.  By Theorem 5.1.4, we know that the Hom of two objects is supported in the intersection of the supports of the two objects.  This is interesting since we can have objects that do not intersect with the the Lagrangian $Z = Z_{\mathrm{reg}}$ yet intersects with $Z_{\mathrm{triv}}$, which means having a trivial category with a nontrivial group action and a nontrivial fixed point category.  Such objects can easily be seen in the abelian case, where $Z_{\mathrm{triv}} = T_1^*T_\CC^\vee$ is the conormal at 1 and $Z_{\mathrm{reg}} = T_{T_\CC^\vee}^*T_\CC^\vee$ is the zero section, any Lagrangian that does not intersect the zero section but intersects the conormal at 1 will do.  A similar phenomenon appears in homotopy theory of spaces, where you have a contractible universal bundle $ES^1$ with the $S^1$-fixed point space not being a point.  We will briefly sketch how to construct a trivial category with a nontrivial $S^1$ action.  

\begin{ex}
Suppose we have a smooth variety $X$ with a superpotential $W: X\rightarrow \CC^\times$.  Then we have a topological $S^1$ action on $\Coh(X)$ defined by
$$\mathcal F \xrightarrow{W} \mathcal F.$$
The fixed point category $\Coh(X)^{S^1}$ is $\Coh(W^{-1}(1))$.  Now suppose we take $X = \CC^\times$, $W = z$, and we consider the category
$$\mathcal C = \Coh(W^{-1}(1)) / \Perf(W^{-1}(1))$$
which is trivial since $W^{-1}(1) = 1$ is smooth and inherits the $S^1$ action from above.  Now if we look at the $S^1$ fixed point category, we are looking at the condition $W^{-1}(1)$ twice and so we are considering
$$\{1\}\times_{\CC^\times} \{1\} = \Spec(\CC[\eta]), |\eta| = -1,$$
which is the based loop space of $\CC^\times$ at 1.  So we have the $S^1$ fixed point category
$$\mathcal C^{S^1} = \Coh(\CC[\eta]) / \Perf(\CC[\eta]),$$
which is nontrivial since the augmentation module is not perfect.
\end{ex}

\section{Sheaves of Categories on $X$: the Conic Story}
\subsection{Singular Support of $\Perf(X)$-modules}

Suppose $X$ is a smooth quasi-compact separated Noetherian scheme, then the $\Perf(X)$ modules $\perfxmod$ is considered as the first approximation of Rozansky--Witten theory of $T^\ast X$.  By taking the singular support of a module, we obtain a conical subset in $T^*X$, and the modules with Lagrangian singular support corresponds to sheaves of categories over conical Lagrangians.  We denote the full subcategory of objects with Lagrangian singular support as $\xwcmod$.  The superscript ``wc'' denotes weakly constructible and it is expected that such a characterization exists analogous to the case in smooth manifolds \cite{KS}.  The notion of singular support of $\Perf(X)$-modules and their properties are due to the work of Stefanich et al \cite{stefanich}.  This definition of singular support is analogous to the singular support of ind-coherent sheaves of \cite{AG} and support for objects in compactly generated triangulated categories of \cite{BIK}.

Given a $\Perf(X)$ module $M$, the endomorphisms of $M$ is acted on by the center of $\Perf(X)$:
$$Z(\Perf(X)) = \Hom_{\Perf(X)\otimes\Perf(X)^{\mathrm{op}}}(\Perf(X),\Perf(X)) = \Coh(\mathcal LX),$$
where $\mathcal LX = X\times_{X\times X} X$ is the loop space of $X$, the self intersection of $X$ in $X\times X$ under the diagonal map.  The structure sheaf of the loop space $\mathcal LX$ is then
$$\mathcal O_{\mathcal LX} = \mathcal O_X\otimes_{\mathcal O_{X\times X}}\mathcal O_X$$
which by the Hochschild--Kostant--Rosenberg theorem \cite{Ben_Zvi_2012} we have
$$\mathcal O_{\mathcal LX} = \mathrm{Sym}_{\mathcal O_X}\Omega_{X}[1].$$
Now the monoidal product of $\Coh(\mathcal LX)$ is given by convolution and its unit is $\mathcal O_X$.  So $\mathcal O_X$ acts on $M$, and so informally by Koszul duality,
$$\Ext^*(\mathcal O_X, \mathcal O_X) = \mathcal O_{T^*X[2]}$$
acts on $M$, so treating $M$ as a sheaf over $T^*X[2]$, we can take its support which is our singular support \cite{stefanich}\cite{nolan}.  Here we mention Koszul duality, remember that for a locally complete intersection closed embedding $i: Z\rightarrow X$, we have a free resolution of $\mathcal O_Z$ over $\mathrm{Sym}_{\mathcal O_Z}\Omega_{Z/X}[1]$ given by
\begin{align*}
\cdots&\rightarrow\mathrm{Sym}_{\mathcal O_Z}^j\Omega_{Z/X}[j]\otimes\mathrm{Sym}_{\mathcal O_Z}\Omega_{Z/X}[1]\rightarrow\cdots\rightarrow\mathrm{Sym}_{\mathcal O_Z}^2\Omega_{Z/X}[2]\otimes\mathrm{Sym}_{\mathcal O_Z}\Omega_{Z/X}[1] \\
&\rightarrow\mathrm{Sym}_{\mathcal O_Z}^1\Omega_{Z/X}[1]\otimes\mathrm{Sym}_{\mathcal O_Z}\Omega_{Z/X}[1]\rightarrow\mathrm{Sym}_{\mathcal O_Z}\Omega_{Z/X}[1]\rightarrow \mathcal O_Z\rightarrow 0.
\end{align*}
This is a free resolution since $Z$ is a locally complete intersection, hence $\Omega_{Z/X}$ is locally free.  A local computation that confirms this is indeed a resolution is given in \cite{gw}.  If we specify $i$ to be the diagonal morphism of $X$, then we have a resolution of $\mathcal O_X$ over $\mathcal O_{\mathcal LX}$.  Notice that the $j$-th term of the resolution $P^j$ has degree $\leq -j$ and it is generated by the degree $-j$ component $P_{-j}^j$.  So $\Hom^i(P^j, \mathcal O_Z)=0$ for $i<j$ \cite[\nopp 2.1.2]{bgs}, as $\mathcal O_Z$ is of degree 0.  We also have $\Hom^i(P^j, \mathcal O_Z)=0$ for $i>j$ since $P_{-j}^j$ is mapped to $0$ and that $P^j$ is generated by the degree $-j$ component.  So we see that there are only degree $j$ maps from $P^j$ to $\mathcal O_Z$, hence
$$\mathrm{Ext}_{\mathrm{Sym}_{\mathcal O_Z}\Omega_{Z/X}[1]}^*(\mathcal O_Z, \mathcal O_Z) = \mathrm{Sym}_{\mathcal O_Z}\Omega_{Z/X}^\vee[-2] = \mathcal O_{T_Z^*X[2]}.$$
Specializing to the diagonal morphism we have $T_Z^*X = T^*X$.
The formal definition of singular support is as follows:

\begin{defn}
\cite{stefanich} Given $M$ in $\perfxmod$, the endomorphisms of $M$ $\End(M)$ is a module of the center of $\Perf(X)$, which is $\Coh(\mathcal{L}X)$.  The unit of the center acts as the unit of $\End(M)$, so an endomorphism of the unit of the center maps to an endomorphism of the unit of $\End(M)$.  Taking cohomology, we have a map
$$\Ext^\ast(\unit_{\Coh(\mathcal{L}X)}) \rightarrow \Ext^\ast(\unit_{\End(M)}).$$
Now
$$\Ext^\ast(\unit_{\Coh(\mathcal{L}X)}) = \Ext_{\Sym_{\mathcal O_X}\Omega_X[1]}^* (\mathcal O_X, \mathcal O_X) = \Sym_{\mathcal O_X} \Omega_X^\vee[-2] = \mathcal O_{T^*X[2]},$$
by Koszul duality, where the unit of $\Coh(\mathcal LX)$ is $\mathcal O_X$ the zero section given by the diagonal.  Taking the support of $\Ext^\ast(\unit_{\End(M)})$ as a $\mathcal O_{T^*X[2]}$ module would be the singular support of $M$ denoted by $SS(M)$.  $SS(M)$ is a conic subset of $T^*X[2]$ because the fibers are graded.
\end{defn}

This is equivalent in saying that $\End(M)$ is a $\Coh(\mathcal LX)$ module, where $\Coh(\mathcal LX)$ is $E_2$ monoidal, and hence in the $E_1$ sense, by Koszul duality, a $\Perf(T^*X[2])$ module, where $\Perf(T^*X[2])$ is $E_\infty$ monoidal \cite{stefanich}.  They are however not in the $E_2$ sense equivalent, as $\Coh(\mathcal LX)$ would be equivalent to a deformation of $\Perf(T^*X[2])$, which we denote by $\Perf_{\hbar}(T^*X[2])$.  We expect (but could not find a reference for) this $\Perf_\hbar(T^*X[2])$ to come from the $P_3$-algebra structure given by the 2-shifted symplectic structure of $T^*X[2]$, which by Kontsevich formality \cite{Kontsevich_1999}\cite{LV} is equivalent to an $E_3$-algebra structure.  This means that we only consider the assignment of a category to an open set without considering the monoidal product, i.e. $\End(M)$ is a topological sheaf of categories over $T^*X[2]$ \cite{stefanich}.  The support of this topological sheaf is precisely the singular support of $M$.

\begin{rmk}
\cite{stefanich} $\QCoh(\mathcal LX)$ is not equivalent to $\QCoh(T^*X[2])$ but rather a full subcategory of it, those that have support at the zero section \cite{AG}, $\QCoh(T^*X[2])$ is instead $\IndCoh(\mathcal LX)$ and they are different because $\mathcal LX$ is singular.  This is why we work with $\Perf(X)$ modules rather than $\QCoh(X)$ modules.
\end{rmk}

We will now calculate the singular support of $\Perf(Z)$ where $Z$ is a smooth subscheme of $X$.  This module plays a similar role of the constant sheaf $k_N$ with stalks $k$ on a closed submanifold $N \subset M$ in the analogous case of sheaves on manifolds.  We see that our definition of singular support does satisfy our intuition of measuring the codirections where changes happen \cite{KS}\cite{101025}, in this case the conormal of $Z$ in $X$ \cite{stefanich}.

\begin{thm}
Suppose we have $i: Z\rightarrow X$ a closed immersion and both $Z$ and $X$ smooth.  Then $\Perf(Z)$ is a $\Perf(X)$ module by $i^*$.  The singular support $SS(\Perf(Z))$ is the conormal $T_Z^*X$.
\end{thm}

\begin{proof}
We have
$$\End_{\Perf(X)}(\Perf(Z)) = \Coh(Z\times_X Z),$$
and that the cohomology of the structure sheaf
$$H^*(\mathcal O_{Z\times_X Z}) = \Sym_{\mathcal O_Z}\Omega_{Z/X}[1] = \mathcal O_{N_ZX[-1]},$$
is the structure sheaf of the shifted normal $N_ZX[-1]$ by the work of Arinkin--C\u ald\u araru \cite{ARINKIN2012815}.  
By the calculation above Definition 5.1.1, we see that the $E_2$ page of the Ext Eilenberg--Moore spectral sequence \cite{derivedcats}\cite{bmr}:
$$E_2^{p,q} = \mathrm{Ext}_{H^*(A)}^{p,q}(H^*(M),H^*(N)) \implies \mathrm{Ext}_A^{p+q}(M,N)$$
degenerates on the $E_2$ page since $E_2^{p,q}=0$ except when $p=q$.  So we have
$$\Ext_{\mathcal O_{Z\times_X Z}}^\ast(\mathcal O_Z, \mathcal O_Z) = \Ext^*_{\Sym_{\mathcal O_Z}\Omega_{Z/X}[1]}(\mathcal O_Z, \mathcal O_Z) = \Sym_{\mathcal O_Z}\Omega_{Z/X}^\vee[-2] = \mathcal O_{T_Z^*X[2]},$$
where the second and last equality comes from the calculation above Definition 5.1.1 and $\mathcal O_{T_Z^*X[2]}$ is the structure sheaf of the shifted conormal $T_Z^*X[2]$.  And we see that the singular support $SS(\Perf(Z))$ is given by its conormal $T_Z^*X$.
\end{proof}

If in addition to $i$ being a closed immersion, the normal bundle of $Z$ extends to a vector bundle on the first infinitesimal neighborhood of $Z$ in $X$, then $\mathcal O_{Z\times_X Z}$ is formal, and that the self intersection is $\Spec_Z(\Sym_{\mathcal O_Z}\Omega_{Z/X}[1])$ the shifted normal \cite{ARINKIN2012815}.  The action is more explicit in this case, we view the shifted normal $N_ZX[-1]$ as $\mathcal L_ZX / \mathcal LZ$, where $\mathcal L_ZX$ is the loops of $X$ based in $Z$ and $\mathcal LZ$ is the loops of $Z$.  $\mathcal LX$ acts naturally on $\mathcal L_ZX / \mathcal LZ$ by the composition of loops:
$$\mathcal LX \times_X (\mathcal L_ZX / \mathcal LZ) \rightarrow \mathcal L_ZX / \mathcal LZ,$$
so we have
$$\Coh(\mathcal LX)\otimes \Coh(\mathcal L_ZX / \mathcal LZ) \rightarrow \Coh(\mathcal L_ZX / \mathcal LZ)$$
\sloppy
by convolution, which gives $\Coh(\mathcal L_Z X/\mathcal LZ)$ a $\Coh(\mathcal LX)$-module structure coming from the composition of loops.  Using Koszul duality, we again see that the singular support $SS(\Perf(Z))$ is given by the conormal $T^*_ZX$.

When we are working with sheaves on a smooth manifold $M$, then there is a functor $\muhom$ \cite{KS}, which is a sheaf on $T^*M$, and its direct image under the projection map is the standard $\Hom$ \cite{KS}.  Now suppose $M, N$ are $\Perf(X)$ modules, then the center of $\Perf(X)$, $\Coh(\mathcal LX)$ acts on $\Hom(M,N)$ as well.  Hence $\Hom(M,N)$ is a topological sheaf of categories over $T^*X[2]$.  So the microlocal hom $\muhom(M,N)$ in our theory is just $\Hom(M,N)$ \cite{stefanich}.  Now we prove some properties that are analogous to the case in smooth manifolds.

\begin{thm}
\sloppy
Suppose $\Ext^*(\mathrm{unit}_{\End(M)})$ and $\Ext^*(\mathrm{unit}_{\End(N)})$ are finitely generated $\mathcal O_{T^*X[2]}$ modules.  Then $\supp\Hom(M,N) \subseteq SS(M) \cap SS(N)$.

\end{thm}

\begin{proof}
Any morphism from $M$ to $N$ can be considered as a composition of endomorphism of $M$ and $M$ to $N$ or $M$ to $N$ and an endomorphism of $N$.  So annihilating $\End(M)$ will annihilate $\Hom(M,N)$ and annihilating $\End(N)$ will annihilate $\Hom(M,N)$.

\end{proof}

\begin{thm}
\sloppy
Suppose $\mathrm{colim}_i M_i$ is a finite colimit in $\perfxmod$ and that $\Ext^*(\mathrm{unit}_{\End(M_i)})$ are finitely generated.  Then $\displaystyle SS(\mathrm{colim}_i M_i) \subseteq \bigcup_i SS(M_i)$.
\end{thm}

\begin{proof}
First we have $\End(\mathrm{colim}_i M_i) = \mathrm{lim}_i \Hom(M_i, \mathrm{colim}_i M_i)$ as $\Perf(T^*X[2])$-modules in the $E_1$ sense, hence a topological sheaf of categories over $T^*X[2]$.  The support of this sheaf is precisely the singular support of $\mathrm{colim}_i M_i$.  Now suppose we have an open $U$ that is outside of $\displaystyle \bigcup_i SS(M_i)$.

For simplicity sake, let us denote $\Hom(M_i, \mathrm{colim}_i M_i)$ by $N_i$.  We have
$$\Perf(U)\otimes_{\Perf(T^*X[2])} \mathrm{lim}_i N_i = (\QCoh(U)\otimes_{\QCoh(T^*X[2])} \mathrm{Ind}(\mathrm{lim}_i N_i))^\omega$$
\sloppy
where $(-)^\omega$ denotes taking compact objects.  It suffices to show that the above is zero.  Now $\QCoh(U)$ is dualizable as a $\QCoh(T^*X[2])$-module in $\mathcal Pr^L$\cite{bzfn2010}, so the functor $\QCoh(U)\otimes_{\QCoh(T^*X[2])} (-) $ preserves finite limits and thus also preserves inclusions.  Even though the functor $\mathrm{Ind}: \mathrm{Cat}_\infty^{\mathrm{perf}}\rightarrow \mathcal Pr^L$ does not preserve limits, we still have an inclusion for finite limits
$$\mathrm{Ind}(\mathrm{lim}_i N_i) \rightarrow \mathrm{lim}_i \mathrm{Ind}(N_i) \footnote{The author thanks German Stefanich for pointing out that $\QCoh(U)$ is not dualizable in $\mathcal Pr_{st, \omega}^L$, which is equivalent to $\mathrm{Cat}_\infty^{\mathrm{perf}}$ under $\mathrm{Ind}$, and for mentioning that we still have an inclusion $\mathrm{Ind}(\mathrm{lim}_i N_i) \rightarrow \mathrm{lim}_i \mathrm{Ind}(N_i)$ for finite limits that is sufficient to run the argument.}.$$  So we have an inclusion
\begin{align*}
\QCoh(U)\otimes_{\QCoh(T^*X[2])}\mathrm{Ind}(\mathrm{lim}_i N_i) &\rightarrow \QCoh(U)\otimes_{\QCoh(T^*X[2])} \mathrm{lim}_i \mathrm{Ind}(N_i)\\
&= \mathrm{lim}_i \QCoh(U)\otimes_{\QCoh(T^*X[2])} \mathrm{Ind}(N_i) \\
&= \mathrm{lim}_i \mathrm{Ind} (\Perf(U)\otimes_{\Perf(T^*X[2])} N_i) \\
&= \mathrm{lim}_i \mathrm{Ind}(0) \\
&= 0,
\end{align*}
where the second last line comes from $\mathrm{supp} N_i \subseteq SS(M_i)$ by (the proof of) Theorem 5.1.4.  So we have the left hand side $\QCoh(U)\otimes_{\QCoh(T^*X[2])}\mathrm{Ind}(\mathrm{lim}_i N_i) = 0$ and hence $\Perf(U)\otimes_{\Perf(T^*X[2])} \mathrm{lim}_i N_i = 0$.
\end{proof}

\subsection{Abelian Gauge Theories and Fourier Transform}
Now that we have a basic understanding of singular supports of $\Perf(X)$ modules, we can now go back to Statement 1.1.1 with $G = T$ abelian \cite{T}:

\begin{repstat}{teleman}
Pure topological gauge theory in 3 dimensions for a compact abelian Lie group $T$ is equivalent to the Rozansky-Witten theory for $T^*T_\CC^\vee$.
\end{repstat}

This statement relating the category of $T$-categories and the category of boundary conditions of $T^*T_\CC^\vee$ can be seen as a Fourier transform \cite{T}\cite{AD}.  We will go over this statement in the simplest case of $G = S^1$ but the argument for general $T$ is the same.  First we will need the definition of the Cartier dual of an affine commutative $k$-group scheme $G$.

\begin{defn}
\cite{C62}\cite{D73}\cite{AM} For an affine commutative $k$-group scheme $G$, the Cartier dual of $G$ is given by
$$G^\vee = \Hom(G, \mathbb G_m),$$
where the $\Hom$ is the internal $\Hom$.  Notice here we use the same notation $G^\vee$ for both the Langlands dual group of $G$ and the Cartier dual of $G$.  But since the Langlands dual group is only used in Statement 1.1.1 and from now on we only care about the abelian version of Statement 1.1.1, we will only mean the Cartier dual of $G$ after this.
\end{defn}

\begin{ex}
\cite{AD}\cite[Appendix A]{DP} The Cartier dual of $\ZZ$ is $\mathbb G_m$ and the Cartier dual of $\mathbb G_m$ is $\ZZ$.
\end{ex}

We will now state a Fourier transform theorem.

\begin{thm}
\cite{AM} Let $G$ be an affine commutative $k$-group scheme.  Then the Fourier--Mukai transform
$$\QCoh(G) \rightarrow \QCoh(BG^\vee)$$
induced by the pairing $G\times BG^\vee\rightarrow B\mathbb G_m$ is an equivalence.  Furthermore, this functor is symmetric monoidal taking convolution in $\QCoh(G)$ to the usual tensor product in $\QCoh(BG^\vee)$.
\end{thm}

Applying this result to $G = \mathbb G_m$ and $G^\vee = \ZZ$, we obtain a Fourier--Mukai equivalence of $\QCoh(\mathbb G_m)$ and $\QCoh(B\ZZ)$.  Now topologically we know that $B\ZZ \simeq S^1$, and so $\QCoh(B\ZZ)$ is the local systems on $S^1$ \cite{AD}.  The equivalence exchanges the convolution product of $\QCoh(\mathbb G_m)$ with the usual tensor product of $\QCoh(B\ZZ)$.  $\QCoh(\mathbb G_m)$ is just vector spaces along with an invertible operator, and the convolution product is just tensor product over $k$, with the operator now being the tensor of operators \cite{AD}.  Now $\QCoh(B\ZZ)$ the category of local systems on $S^1$ is given by the monodromy, i.e. a $\ZZ$ representation, and the tensor of operators in $\QCoh(\mathbb G_m)$ precisely corresponds to the tensor of representations \cite{AD}.

However, if we want topological $S^1$-categories, meaning that the $S^1$ action is continuous, we are looking at modules over $\QCoh(B\ZZ)$ with the convolution product not the usual tensor product \cite{AD}.  What we want instead is an equivalence that exchanges the usual tensor product in $\QCoh(\mathbb G_m)$ and the convolution product in $\QCoh(B\ZZ)$.  This is not proven in full generality with arbitrary $G$ and $BG^\vee$ \cite{AM}, but since we are only concerning $\mathbb G_m^n$ and $B\Lambda$ or rather $\mathbb G_m$ and $B\ZZ$, this is known for example in \cite{Fang_2011}, where the convolution in $\QCoh(B\ZZ)$ corresponds to the tensor product over $k[x, x^{-1}]$, i.e. identifying the operators on either entry.  What we have now instead is the symmetric monoidal equivalence of $\QCoh(\mathbb G_m)$ with the usual tensor product and $\QCoh(B\ZZ)$ with the convolution product.  So we have the equivalence of $\mathrm{Mod}_{(\QCoh(B\ZZ), \ast)}(\mathcal Pr^L_{st})$ and $\mathrm{Mod}_{(\QCoh(\mathbb G_m), \otimes)}(\mathcal Pr^L_{st})$ \cite{AD}.  Now the former are stable presentable categories with a topological $S^1$ action, as $\QCoh(B\ZZ)$ is the category of local systems on $S^1$, any contractible subset of $S^1$ trivializes the fibers, and hence also trivializes the action.  The latter is just the quasicoherent sheaves of categories on $\mathbb G_m$ since $\mathbb G_m$ is 1-affine \cite{AD}.

The idea is that sheaves of categories (or modules over categories of sheaves) should have a notion of singular support and hence these sheaves of categories have support on $T^*\mathbb G_m$.  So the picture outlined in Section 4.3 is more or less seen by the Fourier transform construction.  However, in light of Remark 5.1.2, we see that we should be considering $\Perf(\mathbb G_m)$ modules in $\mathrm{Cat}_\infty^{\mathrm{perf}}$.  We would want a Fourier transform for small categories that relates $\Coh(B\ZZ)$ with its convolution product with $\Perf(\mathbb G_m)$ with its usual tensor product.  However, this is not true, $\Coh(B\ZZ)$ is equivalent to a subcategory of $\Perf(\mathbb G_m)$ \cite[Appendix A]{DP}.  One can think of $\Coh(B\ZZ)$ as coherent local systems on $S^1$, and the monodromy can be decomposed into Jordan blocks, with finite number of eigenvalues, which corresponds to finite support in $\mathbb G_m$.  Now any modules over $\Perf(\mathbb G_m)$ will be a module over its subcategory so for these modules we can still have our notion of singular support on $T^*\mathbb G_m$.  However this means that our story is only a part of this story as the inclusion is only one way.

Now even if we discard any functional analytic issues about large and small categories, there is still one issue.  Namely the singular support in $T^*\mathbb G_m$ is conic, since the singular support lives in $T^*\mathbb G_m[2]$, with graded fibers.  This means that we do not have graphs as our singular support, and there will not be any Lagrangians that do not intersect the zero section yet intersect with the conormal at 1, which we did outline a construction on the other side in Section 4.3.  This shows that this Fourier transform picture is not complete and only captures part of the story.

Let us turn our eyes back to the classical picture of topological spaces and topological group actions.  Given a space with a continuous group action $G$, there are spaces that are homotopy equivalent but not $G$-equivariant homotopy equivalent.  A notable example would be a point and the contractible space $EG$.  The true homotopy theory, or rather genuine homotopy theory would be considering $G$-equivariant continuous maps between $G$-spaces and $G$-equivariant homotopies between such maps.  Elmendorf's theorem \cite{E} states that the genuine homotopy theory of spaces is equivalent to the homotopy theory of diagrams of fixed point spaces.  To state this theorem formally, we first go over the definition of an orbit category.  We follow the treatment in \cite{AD2}.

\begin{defn}
The orbit category $\mathcal O_G$ of a topological group $G$ is a category with objects being $G/H$ where $H$ is a closed subgroup of $G$ and morphisms being $G$-equivariant maps.
\end{defn}

Now given any $G$-space $X$, we have a presheaf on $\mathcal O_G$ that values in the category of topological spaces $\mathrm{Top}$, by taking the fixed point space of a closed subgroup $H$ $X^H$ on the point $G/H$:
$$X \mapsto (X^{(-)}: G/H \mapsto X^H).$$
We call this functor $F$.  The category of $\mathrm{Top}$-valued presheaves on $\mathcal O_G$ $\Pshv(\mathcal O_G, \mathrm{Top})$ has a projective model structure with pointwise weak equivalences and fibrations \cite{AD2}.  The functor $F$ above is a Quillen equivalence, with the left adjoint being the assignment to $G/e$, where $e$ is the identity of the group $G$ \cite{E}\cite{MS2}\cite{AD2}.

\begin{thm}
\cite{E}\cite[Chapter VI]{EHCT}\cite{RJP}\cite{MS2}\cite{AD2} The functor $F$ induces an equivalence of $(\infty,1)$-categories between $\mathrm{GTop}$ the category of $G$-spaces and $\Pshv(\mathcal O_G, \mathrm{Top})$, where the weak equivalences for both categories are specified by a family of subgroups of $G$.
\end{thm}

So learning from this classical Elmendorf's theorem, we would be tempted to think something similar might be happening in our case: that instead of considering $G$-categories, we should be considering diagrams of $G$-categories, perhaps diagrams of fixed point categories.  This statement is true and proven in \cite{GMR} with the assumption that $G$ is a finite simplicial group and the category of $G$ objects in a locally finitely presentable, cofibrantly generated simplicial category (see Assumptions 5.8 in \cite{GMR} to see the full list of assumptions).  This is also true if we consider ordinary categories (not $(\infty, 1)$-categories) with $G$ being a discrete group, proven in \cite{rubin2020elmendorfconstructionsgcategoriesgposets}.  The closest result that suits our purposes is the one in \cite{MS2}, in which Stephan proved an Elmendorf theorem for $G$ a compact Lie group and the category of $G$-objects in a cofibrantly generated topological model category, with the condition that the family of subgroups considered satisfies some cellularity conditions.

To see why these results are relevant to us, first of all we note that both simplicial categories, i.e. categories that are enriched over simplicial sets, and topological categories are models of $(\infty, 1)$-categories \cite{L}.  Secondly, we see that some assumptions are of no issue to us (or are close to that).  Model categories that are locally presented and cofibrantly generated are combinatorial model categories and are equivalent to presentable $(\infty, 1)$-categories \cite{Pavlov_2025}, and since we are working in the environment $\mathrm{Cat}_\infty^{\mathrm{perf}}$ which is presentable, we satisfy one of the assumption for \cite{MS2} (though not \cite{GMR} since the assumption requires locally finitely presentable).  So we have some idea on how to enlarge our story on the side of $G$-categories, namely by considering diagrams of fixed point categories.  On the other side, in the case of $G=S^1$, we would want to consider objects that have nonconic support in $T^*\mathbb G_m$, especially having supports that do intersect the conormal at 1 with no intersection with the zero section.  In the next section, we extend the theory and construct many objects with nonconic support in $T^*\mathbb G_m$.

\section{Sheaves of Categories on $X\times \A^1$: the Nonconic Story}
\subsection{Tamarkin's Trick}
We will now use the conic theory of Section 5 to model nonconic objects by using Tamarkin's trick \cite{10.1007/978-3-030-01588-6_3} (though this technique has appeared in the complex case earlier in \cite{8178551}) and consider sheaves of categories on $X \times \A^1$ instead of $X$.  An introduction to this technique in symplectic geometry can be found in \cite{sq}.  Suppose we are given an exact Lagrangian in $T^*X$ with $\lambda |_L = df$ for some function $f$ where $\lambda$ is the canonical 1-form.  Then we can construct a conic Lagrangian in $T^*(X\times \A^1)$ by scaling the cofibers of $(l, f(l), 1)$:
$$\mathrm{Cone}(L) = \{(x,f(x,\xi),c\cdot \xi, c\cdot 1)\in T^*(X\times\A^1): (x,\xi)\in L, c\in k\}.$$
There is an action on $T^*(X\times \A^1)$ by $\A^1$ given by translation of the $\A^1$ base coordinate, which corresponds to different choices of $f$, since the choice of $f$ is up to a constant.  This action is Hamiltonian, with moment map $\mu:T^*(X\times\A^1)\rightarrow (\A^1)^\vee$ given by
$$\mu(x,y,\xi,\tau) = \tau,$$
the projection to the $T^*\A^1$ cofiber.  The symplectic reduction $\mu^{-1}(1^\vee)/\A^1$ will give us back the exact Lagrangian in $T^*X$.  To micmic this construction, we define the category $\mathcal C^{pre}(X)$ as follows:

\begin{defn}
$$\mathcal C^{pre}(X) = \perfxamod / \{ M \in \perfxamod: SS(M) \subset \{\tau = 0\}\},$$
where $T^*(X\times \A^1)$ is given by $(x,y,\xi,\tau)$ with latin alphabets indicating base coordinates and greek alphabets indicating fiber coordinates.
\end{defn}

\begin{defn}
We define the map
$$\rho: T^*X \times T_{\tau\neq 0}^*\A^1 \rightarrow T^*X, \quad (x,\xi,y,\tau) \mapsto (x,\frac{\xi}{\tau}).$$
The microsupport of a module $M$ where $SS(M)$ is not contained in $\{\tau=0\}$ would then be $\musupp(M) = \rho(SS(M) \cap \{\tau\neq 0\})$.  Note that in our context microsupport $\musupp(M)$ refers to the nonconic support in $T^*X$ and singular support $SS(M)$ refers to the conic support in $T^*(X\times \A^1)$.
\end{defn}

Definitions 6.1.1 and 6.1.2 are similar to the case in symplectic geometry, except the condition $\tau > 0$ is replaced by $\tau \neq 0$ \cite{10.1007/978-3-030-01588-6_3}\cite{Ike_2019}\cite{sq}.

\begin{rmk}
The definition of $\mathcal C^{pre}(X)$ in Definition 6.1.1 corresponds to the construction of $\mu^{-1}(1^\vee)$ earlier.  Recall that our singular support $SS(M)$ belongs to $T^*(X\times\A^1)[2]$ instead of $T^*(X\times\A^1)$, so the fibers are graded and hence it is impossible for us to impose the condition $\tau = 1$.  Instead what we do is to require that $\tau \neq 0$.
\end{rmk}

What we are doing here is microlocalizing away from $\tau = 0$, which parallels the construction in \cite{benzvi2025potentcategoricalrepresentations}, where Ben-Zvi and Nadler work in large categories instead of small categories.  They first construct a periodic base
$$\mathcal A = 2\IndCoh(\GG_a) / 2\QCoh(\GG_a)$$
which is symmetric monoidal with convolution given by the group structure of $\GG_a$, where $2\IndCoh$ is the ind-coherent sheaves of categories.  Then they define a periodization
$$2\IndCoh^\pi(X) = 2\IndCoh(X)\otimes \mathcal A,$$
as an $\mathcal A$ module.  The authors explained that $\mathcal A$ is the microlocalization away from the zero section in $T^*\GG_a$, and $2\IndCoh^\pi(X)$ is the parallel microlocalization of $2\IndCoh(X \times \GG_a)$, which is exactly what we are doing here modulo the fact that $2\IndCoh$ is different from modules of $\IndCoh$ with most objects not being able to be recovered from global sections and the size difference of categories we are working with.

Here we first provide an example of a $\Perf(X\times\A^1)$ that is quotiented out:
\begin{ex}
(Non-example): Suppose we have a closed immersion $Z\rightarrow X$, with both $Z$ and $X$ smooth.  We consider $\Perf(Z\times \A^1)$ as a $\Perf(X\times \A^1)$ module.  Then by Theorem 5.1.3, $SS(\Perf(Z\times\A^1))$ is $T_{Z\times\A^1}^*(X\times\A^1)$, and we see that the conormal directions are entirely in the cofiber directions of $T^*X$ and hence entirely lie in $\{\tau=0\}$.
\end{ex}

Now we will see that we recover Theorem 5.1.3 in our construction:
\begin{ex}
Again we are given $Z\rightarrow X$ a closed immersion, with both $Z$ and $X$ smooth.  We consider the module $\Perf(Z)\boxtimes \Perf(0)$.
$$\ind_{p_1}\Perf(Z) = \Perf(X\times\A^1)\otimes_{\Perf(X)}\Perf(Z)=\Perf((X\times\A^1)\times_X Z) = \Perf(Z\times \A^1)$$
and
$$\ind_{p_2}\Perf(0) = \Perf(X\times\A^1)\otimes_{\Perf(\A^1)}\Perf(0)=\Perf((X\times\A^1)\times_{\A^1}0) = \Perf(X\times 0).$$
So 
\begin{align*}
\Perf(Z)\boxtimes\Perf(0) &= \Perf(Z\times \A^1)\otimes \Perf(X\times 0)) \\
&= \Perf((Z\times\A^1) \times_{X\times\A^1} (X\times 0)) \\
&= \Perf(Z\times0).
\end{align*}
The singular support of $\Perf(Z)\boxtimes\Perf(0)$ is then the conormal of $Z\times0$ in $X\times\A^1$ which gives us $\musupp$ the conormal $T^\ast_Z X$.

\end{ex}

We will now go over some examples that will give nonconic microsupport:
\begin{ex}
Suppose $X=\A^n$.  Consider $Z=\{y+f(x)=0\} \subset \A^{n+1}$, with $f$ a polynomial.  Then the singular support of $\Perf(Z)$ is the conormal
$$\{(x_i,-f(x),c\frac{\partial f}{\partial x_i},c): x_i,c\in\A^1\}\subset T^*(\A^{n+1}).$$
So $\musupp\Perf(Z)$ is $\Gamma_{df}$, the graph of $df$.
\end{ex}

\begin{ex}
Now suppose $X=\GG_m$.  And similarly consider $Z=\{y+f(x)=0\} \subset \GG_m \times \A$, with $f$ a polynomial.  Then the singular support of $\Perf(Z)$ is the conormal
$$\{(x,-f(x),c\frac{df}{dx},c): x\in\GG_m, c\in\A\}\subset T^*(\GG_m\times\A^1).$$
So $\musupp\Perf(Z)$ is $\Gamma_{df}$.  Specializing $\displaystyle f = \frac{x^{n+1}}{n+1}$ gives $x^ndz$ and $\displaystyle f = -\frac{1}{nx^n}$ gives $\displaystyle \frac{dx}{x^{n+1}}$.  In particular, we have constructed Lagrangians that do not intersect with the zero section, but intersects with the conormal at 1, which under Statement 1.1.1, corresponds to trivial categories that have nontrivial fixed point categories, which means we have trivial categories with nontrivial $S^1$ action.
\end{ex}

\begin{rmk}
Notice that the differential dx/x is not exact, so one cannot obtain it this way.  If we are working algebraically, we would not be able to obtain it locally as well, since kahler differentials are not locally exact.
\end{rmk}

\begin{ex}
Now suppose $N = \{f_1 = \cdots = f_m = 0\}$ a smooth subvariety of $\A^n$.  Then the conormal of $N$
$$T_N^*\A^n = \{(x,\xi): f_1(x) = \cdots = f_m(x) = 0, \xi = \xi_1(df_1)_x + \cdots + \xi_m(df_m)_x\}.$$
Suppose $\alpha \in \Omega^1(N)$.  Then
$$\alpha + T_N^*\A^n = \{(x,\xi): f_1(x) = \cdots = f_m(x) = 0, \xi = \xi_1(df_1)_x + \cdots + \xi_m(df_m)_x + \alpha_x\}.$$
When $\alpha$ is closed, this is a Lagrangian.  When $\alpha$ is exact, the Lagrangian is exact \cite{ns}.  So suppose $g \in k[x_1,\dots,x_n]/(f_1,\dots,f_m)$ a function on $N$.  Then
$$L = dg + T_N^*\A^n = \{(x,\xi): f_1(x) = \cdots = f_m(x) = 0, \xi = \xi_1(df_1)_x + \cdots + \xi_m(df_m)_x + (dg)_x\},$$
is an exact Lagrangian, where ${\lambda|}_L = dg$.  Here we also use $g$ to denote the map from $L$ to $k$, given by $(x,\xi) \mapsto g(x)$.  The choice of lift of $g$ in $k[x_1,\dots, x_n]$ does not matter, since the $df_i$ terms is included in the previous terms.  We now consider $Z = \{y+g(x) = f_1(x) = \cdots = f_m(x) = 0\}$ which is smooth in $\A^{n+1}$.  The singular support of $\Perf(Z)$ is the conormal
$$\{(x,-g(x),c\xi_1df_1+\cdots+c\xi_mdf_m+cdg,cdy):c,\xi_i\in\A^1, f_1(x) = \cdots = f_m(x) = 0\}.$$
The $\musupp$ of $\Perf(Z)$ will then be $L$.  This construction also works in $\GG_m^n$ where the $g$ and $f_i$'s are Laurent polynomials.
\end{ex}

All of these can be generalized to where $X$ is smooth scheme.  We define objects
$$C_f = \Perf(y+f=0)$$
where $f: X\rightarrow \A^1$ as in Example 6.6, and so $SS(C_f)$ would then be the conormal of $\{y+f=0\}$:
$$SS(C_f) = \{(x,-f(x),c\cdot df, c)\},$$
and hence the microsupport $\musupp(C_f) = \Gamma_{df}$, the graph of $df$.

To model Rozansky--Witten theory, the Hom of such objects are expected to be the matrix factorizations, with support at the intersection of the graphs of $df$ and $dg$ \cite{KR}. 

\begin{thm}
$$\Hom_{\mathcal C^{pre}(X)}(C_f,C_g) = \Coh(X_0)/\Perf(X_0),$$
where $X_0 = X \times_{\A^1} 0$, the fiber at $0$.
\end{thm}

\begin{proof} 
We have 
$$\Hom_{\Perf(X\times\A^1)}(C_f,C_g) = \Coh(\{y+f=0\}\times_{X\times\A^1}\{y+g=0\}).$$
Working locally by taking open affine $U=\Spec A$,
\begin{align*}
\{y+f=0\}\times_{U\times\A^1}\{y+g=0\} &=\Spec( A[y]/(y+f) \otimes_{A[y]} A[y]/(y+g) )\\
&= \Spec\left(\left( \begin{array}{l}
			A[y,\eta], \\
			|\eta| = -1, d\eta = y+f
			\end{array}\right) \otimes_{A[y]} A[y]/(y+g)\right) \\
&= \Spec\left(\begin{array}{l}
			A[\eta], |\eta|=-1,\\
			d\eta = f - g
			\end{array}\right) \\
&= \Spec A \times_{\A^1} 0.
\end{align*}

So we see that
$$\Hom_{\Perf(X\times\A^1)}(C_f,C_g) = \Coh(X_0).$$

Now the $\Hom$ in $\C^{pre}(X)$ is the restriction of the topological sheaf $\Hom_{\Perf(X\times\A^1)}$ to $\{\tau \neq 0\}$, i.e. away from the zero section of the $(\A^1)^\vee$ direction.  By \cite[Proposition 4.2]{10.1093/imrn/rns125} and \cite[Appendix H]{AG}, we see that formally inverting $\tau$ corresponds to quotienting the objects that have support on the zero section, and this gives the category of singularities of $X_0$, i.e.
$$\Hom_{\mathcal C^{pre}(X)}(C_f,C_g) = \Perf(\{\tau \neq 0\}) \otimes_{\Perf(T^*(X\times \A^1)[2])} \Coh(X_0) = \Coh(X_0)/\Perf(X_0).$$
\end{proof}

\begin{rmk}
We can also see that $\mathcal L_0\A^1$ the based loops of $\A^1$ at $0$ acts on $X_0$ by loop composition and loop addition and both are equivalent by the Eckmann--Hilton argument \cite{preygel2011thomsebastianidualitymatrix}.  This addition passes through the equivalence of $X_0$ and $\{y+f(x)=0\}\times_{X\times\A^1}\{y+g(x)=0\}$.  The $\tau$ action that comes from $\mathcal L\A^1$ is equivalent to the loop addition and inverting $\tau$ would give the category of singularities \cite{preygel2011thomsebastianidualitymatrix}.
\end{rmk}

Similarly we can consider the objects as in Example 6.1.9
$$C_{f,Z} = \Perf(y+f = 0, Z \, \,\mathrm{smooth}\,\,\mathrm{subscheme}),$$
with $\musupp(C_{f,Z}) = df + T^*_ZX$.

\begin{thm}
\label{hom_thm}
$$\Hom_{\mathcal C^{pre}(X)}(C_{f,Y},C_{g,Z}) = \Coh((Y\times_X Z)_0) / \Perf(Y\times_X Z)_0),$$
the category of singularities of the fiber $Y\times_X Z \rightarrow \A^1$ at 0.
\end{thm}

\begin{proof}
We have a similar calculation:
$$\Hom_{\Perf(X\times\A^1)}(C_{f,Y},C_{g,Z}) = \Coh(\{y+f=0\,\, \mathrm{in}\,\,Y\}\times_{X\times\A^1}\{y+g=0 \,\,\mathrm{in}\,\,Z\}).$$
Working locally on $U = \Spec A$, such that $Y$ and $Z$ are locally given by $I_Y$ and $I_Z$ respectively and that $f_i$ generates $I_Y$, we have
\begin{align*}
& \{y+f=0\,\, \mathrm{in}\,\,Y\times\A^1\}\times_{X\times\A^1}\{y+g=0 \,\,\mathrm{in}\,\,Z\times\A^1\} \\
&= \Spec(A[y] / (I_Y, \, y+f) \otimes_{A[y]} A[y] / (I_Z,\, y+g)) \\
&= \Spec\left(\left(\begin{array}{l}
			A[y,\epsilon_i,\eta], \\
			|\epsilon_i| = |\eta| = -1, d\epsilon_i = f_i, d\eta = y+f
			\end{array}\right) \otimes_{A[y]} A[y] / (I_Z, \, y+g)\right)\\
&= \Spec\left(\begin{array}{l}
			(A / I_Z)[\epsilon_i,\eta], |\epsilon_i| = |\eta| = -1,\\
			d\epsilon_i = f_i, d\eta = f-g
			\end{array}\right) \\
&= \Spec(A/I_Y \otimes_A A/I_Z) \times_{\A^1} 0.
\end{align*}
So we see that 
$$\Hom_{\Perf(X\times\A^1)}(C_{f,Y},C_{g,Z}) = \Coh((Y\times_X Z)_0),$$
where the map $f-g$ are given by $Y\times_X Z\rightarrow Y \xrightarrow{f} \A^1$ and $Y\times_X Z\rightarrow Z \xrightarrow{g} \A^1$.  Again, we have
$$\Hom_{\mathcal C^{pre}(X)}(C_{f,Y},C_{g,Z}) = \Coh((Y\times_X Z)_0) / \Perf(Y\times_X Z)_0),$$
the category of singularities of the fiber.
\end{proof}

\begin{rmk}
By Theorem 5.1.7, we see that
\begin{align*}
\supp \Hom_{\C^{pre}(X)}(C_f, C_g) &\subseteq SS(C_f)\cap SS(C_g) \cap \{\tau\neq0\} \\
&= \{(x,-f(x),c\cdot df,c): f(x)=g(x), df = dg, c\neq0\}.
\end{align*}
This is not quite right.  The additional condition of $f=g$ is due to the fact that given a Lagrangian, there is an $\A^1$ worth of choices of functions $f: L\rightarrow k$, each differ one another by a shift.  Up until now, we have not identified objects that differ by translation, i.e. we have not quotiented out the $\A^1$ action.
\end{rmk}

We first define the map
$$T_s: X\times \A^1 \rightarrow X\times \A^1, \quad (x,y)\mapsto (x,y+s),$$
with $s$ closed point of $\A^1$.

\begin{defn}
We define $\mathcal C(X)$ with objects $C_{f,Z}$ and their finite colimits.  The $\musupp$ of these objects give $df + T_Z^*X$.  When $f=0$, this gives the conormal $T_Z^*X$ and when $Z=X$, this gives the graph $\Gamma_{df}$.

We define the Hom's in $\mathcal C(X)$ to be
$$\Hom_{\mathcal C(X)}(M,N) = \bigoplus_{s\in k} \Hom_{\mathcal C^{pre}(X)}(M, \res_{T_s} N).$$
If $F\in \Hom_{\mathcal C^{pre}(X)}(M, \res_{T_a}N)$ and $G\in \Hom_{\mathcal C^{pre}(X)}(N, \res_{T_b}L)$ then we $GF = \res_{T_a}(G)F \in \Hom_{\mathcal C^{pre}(X)}(M,\res_{T_{a+b}}L)$.  This construction is known as orbit categories in \cite{keller2005}\cite{fan2025dgenhancedorbitcategories}, skew categories in \cite{0663b65d-0d4f-3e31-9254-503dec3a6488} and $\mathcal{SC}$ in \cite{Chen_2024}.
\end{defn}

\begin{rmk}
Since $\A^1$ acts on $X\times \A^1$ by translation, we have an action
$$\Coh_c(\A^1)\otimes \Perf(X\times\A^1)\rightarrow \Perf(X\times\A^1),$$
and hence also an action its modules, where $\Coh_c(\A^1)$ is the category of coherent sheaves on $\A^1$ with finite support with convolution as its monoidal product.  A larger part of the story should be described by the coinvariants construction \cite{T} given by
$$\mathcal C^{pre}(X) \otimes_{\Coh_c(\A^1)} \mathrm{f.d.}\,\Vect.$$
We should really think $\mathcal C(X)$ as a subcategory of the coinvariants, and the Hom construction above given by direct sum which is really a colimit would then be the Hom in the coinvariants restricted to the objects mentioned above.  In \cite{benzvi2025potentcategoricalrepresentations}, instead of working with coinvariants and modding out the action, Ben-Zvi and Nadler remembers the $\GG_a$ action by working with the periodization as a module of the periodic base with the convolution symmetric monoidal structure.
\end{rmk}

\begin{lem}
\cite{fan2025dgenhancedorbitcategories}\cite{Chen_2024}\cite{0663b65d-0d4f-3e31-9254-503dec3a6488} The objects $\res_{T_s}M$ are isomorphic in $\C(X)$.
\end{lem}

\begin{proof}
Let $\mathcal D^{pre}(X)$ to be the full subcategory of $\C^{pre}(X)$ generated by objects $C_{f,Z}$ and their finite colimits.  We have a canonical map $Q: \mathcal D^{pre}(X)\rightarrow \C(X)$ by sending objects to objects, and 
$$f\in\Hom_{\C^{pre}(X)}(M,N)\subseteq \bigoplus_{s\in k}\Hom_{\C^{pre}(X)}(M,\res_{T_s}N).$$
Consider
\begin{align*}
\id_{M,\C^{pre}(X)}\in\Hom_{\C^{pre}(X)}(M,M)&=\Hom_{\C^{pre}(X)}(M,\res_{T_{-a}}\res_{T_a}M) \\
&\subseteq \Hom_{\C(X)}(M,\res_{T_a}M),
\end{align*}
and
$$\id_{\res_{T_a}M,\C^{pre}(X)}\in \Hom_{\C^{pre}(X)}(\res_{T_a}M,\res_{T_a}M)\subseteq \Hom_{\C(X)}(\res_{T_a}M,M).$$
Then we see that
\begin{align*}
\id_{\res_{T_a}M,\C^{pre}(X)}\id_{M,\C^{pre}(X)} &= \res_{T_{-a}}(\id_{\res_{T_a}M,\C^{pre}(X)})\id_{M,\C^{pre}(X)} \\
&= \id_{M,\C^{pre}(X)}\in \Hom_{\C^{pre}(X)}(M,M),
\end{align*}
which is $\id_{M,\C(X)}\in\Hom_{\C(X)}(M,M)$.  Similarly we have
$$\id_{M,\C^{pre}(X)}\id_{\res_{T_a}M,\C^{pre}(X)} = \id_{\res_{T_a}M,\C^{pre}(X)} = \id_{\res_{T_a}M,\C(X)} \in \Hom_{\C(X)}(\res_{T_a}M,\res_{T_a}M).$$
\end{proof}

Now we revisit Theorem 6.1.10 and Remark 6.1.12 and see that the Hom between $C_f$ and $C_g$ are given by the (augmented) category of matrix factorizations.
\begin{thm}
\label{mf_thm_simple}
$\Hom_{\C(X)}(C_f,C_g) = \MF(X,g-f)$.
\end{thm}

\begin{proof}
Using the definition of $\Hom$ in $\C(X)$, we have
\begin{align*}
\Hom_{\C(X)}(C_f,C_g) &= \bigoplus_{s\in k} \Hom_{\C^{pre}(X)}(C_f,\res_{T_s}C_g) \\
&= \bigoplus_{s\in k} \Hom_{\C^{pre}(X)}(C_f,C_{g-s}) \\
&= \bigoplus_{s\in k} \Coh(X_{s}) / \Perf(X_{s}) \\
&= \MF(X, g-f),
\end{align*}
where $X_s$ is the fiber of $g-f$ at $s$.
\end{proof}

\begin{thm}
$\rho(\supp\Hom_{\C(X)}(C_f,C_g)\cap\{\tau\neq0\})\subseteq \Gamma_{df}\cap \Gamma_{dg}.$
\end{thm}

\begin{proof}
\begin{align*}
\supp\Hom_{\C(X)}(C_f,C_g) \cap\{\tau\neq0\} &= \bigcup_{s\in k} \supp\Hom_{\C^{pre}(X)}(C_f,C_{g-s}) \cap\{\tau\neq0\} \\
&\subseteq \bigcup_{s\in k} SS(C_f) \cap SS(C_{g-s}) \cap \{\tau\neq0\} \\
&= \{(x, y, cdf, c): df = dg, c\neq0\}.
\end{align*}
So $\rho(\supp\Hom_{\mathcal C(X)}(C_f,C_g)\cap\{\tau\neq0\}) \subseteq \Gamma_{df}\cap\Gamma_{dg}$.
\end{proof}

We also calculate the Hom between the general objects of $\mathcal C(X)$ and see that they are given again by augmented matrix factorizations.

\begin{thm}
\label{mf_thm}
$\Hom_{\mathcal C(X)}(C_{f,Y},C_{g,Z}) = \mathrm{MF}(Y\times_X Z, g-f)$.
\end{thm}

\begin{proof}
By the same calculation of Theorem 6.1.17 and Theorem 6.1.12.
\end{proof}

\begin{rmk}
When $Y\times_X Z$ is not regular, the $\Hom$ is not the category of B-branes of Kontsevich \cite{Orlov_2003}.
\end{rmk}

\subsection{The Case of $X = \pt$}
Now we look at the case where $X = \pt$.  $\C(\pt)$ is generated by the objects $\Perf(y=c)$ which we will abbreviate as $\Perf(c)$.  Let $\mathfrak U = \Perf(k[u,u^{-1}])$, which is the 2-periodization of vector spaces.  We have the following theorem:

\begin{thm}
$\mathcal C(\mathrm{pt}) = (\mathfrak U$-\emph{mod}$)^\omega$, the compact objects of $\mathfrak U$-\emph{mod}.
\end{thm}

\begin{proof}
We define a functor from the category $\mathfrak{U}$-mod to $\C(\pt)$ by the map
$$\mathfrak{U} \mapsto \Perf(0).$$
We define the functor from $\C(\pt)$ to $\mathfrak{U}$-mod by the map
$$M \mapsto \Hom_{\C(\pt)}(\Perf(0),M).$$
This map is well defined by the following:
\begin{align*}
&\Hom_{\C(\pt)}(\Perf(0),\Perf(c)) \\
&= \bigoplus \Hom_{\C^{pre}(X)}(\Perf(0), \Perf(c+s)) \\
&= \bigoplus \Perf(k[x,u,u^{-1}]) \otimes_{\Perf(k[x,u])} \Hom_{\Perf(\A^1)}(\Perf(0),\Perf(c+s)) \\
&=\bigoplus \Perf(k[x,u,u^{-1}]) \otimes_{\Perf(k[x,u])} \Coh(\{0\}\times_{\A^1} \{c+s\}) \\
&=\Perf(k[x,u,u^{-1}])\otimes_{\Perf(k[x,u])} \Coh(\{0\}\times_{\A^1}\{0\}) \\
&=\Perf(k[x,u,u^{-1}]) \otimes_{\Perf(k[x,u])} \Coh(k[\eta]), |\eta| = -1
\end{align*}
By Koszul duality, we have
$$\Coh(k[\eta]) = \Perf(k[u]), |u| = 2,$$
\sloppy
with duality exchanging the convolution in $\Coh(k[\eta])$ and the usual tensor product in $\Perf(k[u])$, and that the convolution monoidal unit $k[\eta]/(\eta)$ corresponds to the tensor monoidal unit $k[u]$.
So we have
$$\Hom_{\C(\pt)}(\Perf(0),\Perf(c)) = \Perf(k[x,u,u^{-1}]) \otimes_{\Perf(k[x,u])} \Perf(k[u])$$
by Koszul duality, and this is equivalent to
$$\Perf(k[x,u,u^{-1}]) \otimes_{\Perf(k[x,u])} \Perf(k[u]) = \Perf(k[u,u^{-1}]).$$
Composing these functors we see that these two categories are equivalent.
\end{proof}

\printbibliography

\end{document}